\documentclass[11pt,a4paper,twoside]{article}

\usepackage[a4paper]{geometry}
\usepackage{amsfonts}
\usepackage{amsthm}
\usepackage{amsmath}
\usepackage{parskip}
\usepackage{mathrsfs}
\usepackage{graphicx}
\usepackage{amssymb}
\usepackage{xtab}

\newcommand{\TheAuthor}{Sicheng Zhang}
\newcommand{\TheSupervisor}{Simon L. Rydin Myerson}
\newcommand{\TheTitle}{Fourier Restriction: From Linear Restriction to Multilinear Restriction}

\newcommand{\TheModule}{\bf MA4K9 Dissertation}

\newcommand{\TheUni}{The University of Warwick}
\newcommand{\TheDept}{Mathematics Institute}

\usepackage{setspace}
\usepackage[nodayofweek]{datetime} 

\usepackage{mathabx}
\usepackage[english]{babel}
\usepackage{bbm}

\usepackage{url}

\DeclareMathOperator{\supp}{supp}
\DeclareMathOperator{\nei}{Nei}
\DeclareMathOperator{\vol}{Vol}
\DeclareMathOperator{\BR}{BR}
\DeclareMathOperator{\MR}{MR}
\DeclareMathOperator{\margin}{margin}

\newtheorem{RT}{Restriction Theorem}
\newtheorem{conjecture}{Conjecture}
\newtheorem{definition}{Definition}[section]
\newtheorem{theorem}{Theorem}[section]

\newtheorem{lemma}[theorem]{Lemma}

\begin{document}

\pagenumbering{roman}
\begin{titlepage}
\begin{center}
\includegraphics[width=5cm]{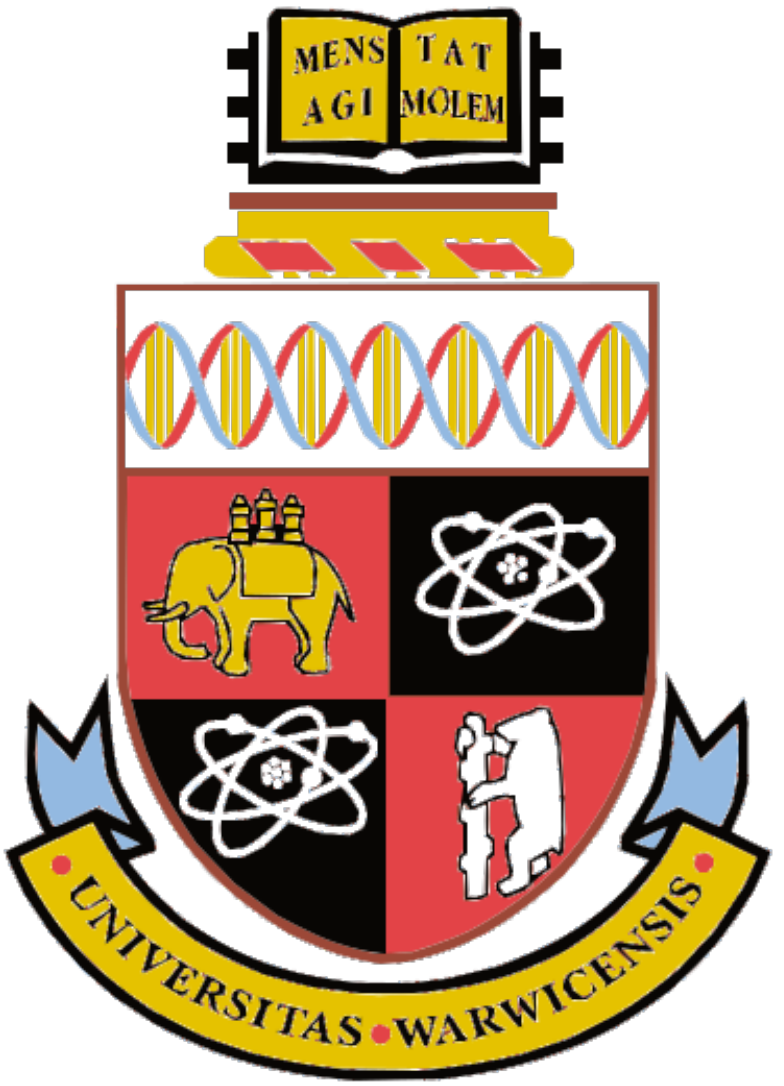} 

\begin{spacing}{2}
\begin{center}
{\Large \bf \TheTitle} 

by

{\Large \bf \TheAuthor} 

Supervised by 

{\Large \bf \TheSupervisor}
\vspace*{8pt}

{\large \bf \TheModule} 

Submitted to \TheUni 

\vspace*{8pt}
{\Large \bf \TheDept} 

April, 2025

\vspace*{8pt}
\includegraphics[width=5cm]{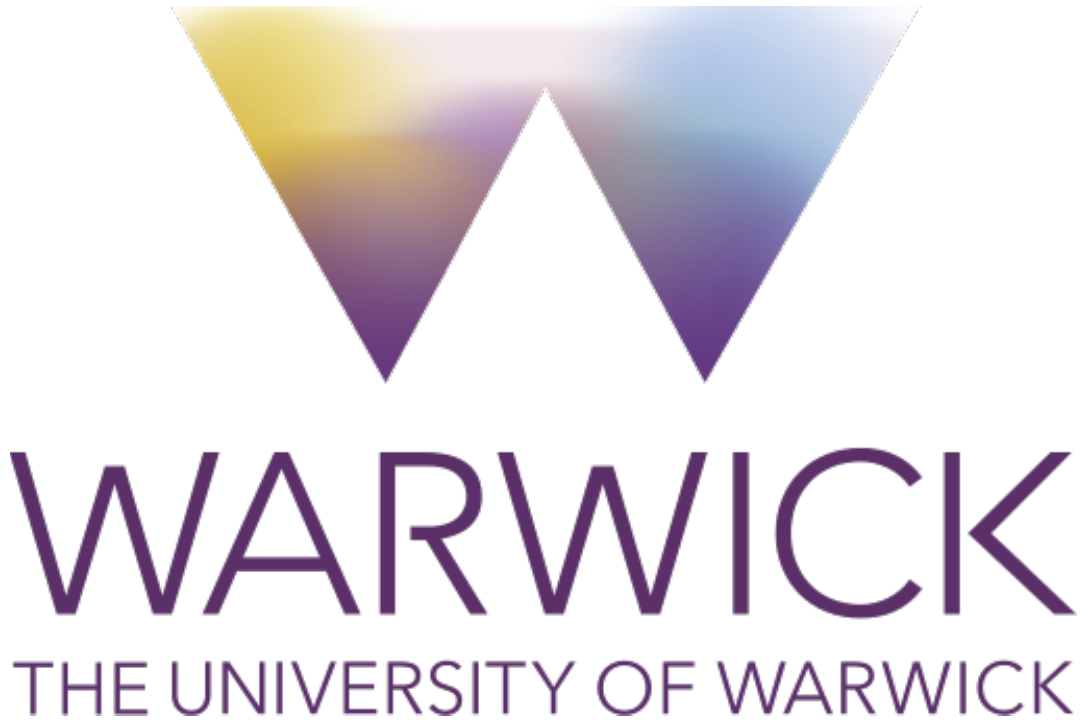} 

\end{center}
\end{spacing}
\end{center}
\end{titlepage}

\setcounter{page}{2}
\tableofcontents
\cleardoublepage
\onehalfspacing
\pagenumbering{arabic}
\onehalfspacing
\raggedbottom

%

\section{Introduction}

This dissertation studies the Fourier restriction, which is to find the range of the constants \(p\), \(q\) such that the \(L^q\) norm on a chosen subset of the Fourier domain is bounded above by the \(L^p\) norm in the whole spatial domain, up to some constant that is independent of the function. Before the inception of this subfield of harmonic analysis, the earliest conclusion of such kind of estimates is the boundness property of the Fourier transform, where \(q=\infty\) and \(p=1\). Since the Fourier transform has its explicit expression only when the function is in \(L^1\), the functions that will be used in restriction estimates will be the Schwartz functions, defined at the start of the next section. Surprisingly, the restriction estimate is much easier on the whole space \(\mathbb{R}^n\), where the range of \(p, q\) are proved and are strict. However, on other manifolds such as the sphere and the paraboloid, the restriction estimates are only partially proved in lower dimensions, and higher dimension cases remain as conjectures. This is because the range of \(p\) and \(q\) in the restriction estimates are larger when the manifolds are specifically chosen, which drastically increases the complexity of the proofs.

The previous restriction problem is classified as linear restriction, which involves a single Schwartz function. There are more generalized cases, which are called bilinear or even multilinear restriction. Those estimates involve the product of multiple functions, and the problem aims to bound its \(L^q\) norm above via the product of \(L^p\) norms up to a constant similar as the linear case. The bilinear and multilinear restriction help improve the linear restriction in the sense that they can be used to prove stricter estimates of the linear case than direct proofs.

The Fourier restriction also has a major contribution to the theory of decoupling, which has various applications to other fields of mathematics such as number theory and PDEs.

This dissertation aims to discuss purely around the theory of Fourier restriction. There will be three main sections. Section 2 is about linear restriction, where the restriction on the whole space is discussed. A proof of the restriction estimate on curves will be presented, and further discussions on the restriction problem on the sphere and paraboloid via the Stein-Tomas argument. Section 3 is about bilinear restriction, where the estimate on 2-dimensional case is proved and used to verify the restriction conjecture on the 2-dimensional paraboloid. Section 4 is about multi-linear restriction, and the entire section is used to prove a close result of the multilinear restriction estimate.

Some notations and definitions which will be used throughout this dissertation:
\begin{itemize}
\item Denote the expectation of a randomized function as: \(\mathbb{E}f.\)
\item Denote the open ball of radius \(r\) on \(x\) as: \(B_r(x):=\{y\in\mathbb{R}^n:\|y-x\|\leqslant r\}\).
\item Denote the closure of a set \(A\) as \(\overline{A}\), denote the complement of a set \(A\) as \(A^c\).
\item Denote the support of a function \(f:\Omega \rightarrow R\): \(\supp(f):=\{x\in\mathbb{R}^n:f(x)\neq0\}\).
\item For two quantities \(X\), \(Y\), denote \(X\lesssim_{par}Y\), or abbreviated as \(X\lesssim Y\), if there exists \(C\) depends on the set of parameters \(par\) such that \(X \leqslant CY\). If also \(Y\lesssim_{par}X\), then denote \(X\sim Y\).
\item Define the normed space \(L^p(\mathbb{R}^n):=\{f(x):(\int_{R^n} |f(x)|^p)^\frac{1}{p}<\infty\}\), with norm \(\|f(x)\|_{L^p(R^n)}:=(\int_{R^n} |f(x)|^p)^\frac{1}{p}\).
\item For \(k\geqslant1\), define the vector space \(C^k(\Omega):=\{f:\Omega\to\mathbb{R}^n:f\) is \(k\) times continuously differentiable\(\}\). And for \(k=\infty\), \(C^\infty(\Omega)\) is defined as the intersection of all \(C^k(\Omega)\) spaces.
\item Define the dual of \(p\) as \(p'\) such that \(\frac{1}{p}+\frac{1}{p'}=1\), \(1\leqslant p<\infty\).
\item For a smooth function \(\psi:U\to\mathbb{R}^{n-d}\), \(U\) is assumed to be a bounded subset of \(\mathbb{R^d}\) throughout this dissertation, define the d-dimensional manifold in \(\mathbb{R}^n\) as: \(\mathcal{M}^\psi:=\{(\xi,\psi(\xi)):\xi\in U\}\).
\item Denote the Lebesgue measure as: \(\mu\). All measures are Lebesgue measure if not specified.
\item All indices in this dissertation will not use the symbol \(i\), and \(i\) only refers to the imaginary unit. Also, \(n\) only refers to the dimension of the chosen space.
\item Denote \(f(x)=O(g(x))\) if there exists a constant \(M>0\) such that \(|f(x)|\leqslant M|g(x)|\).
\end{itemize}

\section{Linear Restriction}
We start with the definition of the Fourier transform. The Fourier transform of a function in \(L^1(\mathbb{R}^n)\) is defined by \(\widehat{f}(\xi):=\int_{\mathbb{R}^n}f(x)e^{-2\pi ix\cdot \xi}dx\), and the inverse Fourier transform for a function in \(L^1(\mathbb{R}^n)\) as \(\widecheck{f}(x):=\int_{\mathbb{R}^n}f(\xi)e^{2\pi ix\cdot \xi}d\xi\). The Fourier transform is bounded in the sense that \(\|\widehat{f}\|_{L^\infty}\leqslant\|f\|_{L^1}\). Another significant relation between a function and its Fourier transform is Plancherel’s identity, such that for any \(f\in L^1\cap L^2\), \(\|\widehat f\|_{L^2}=\|f\|_{L^2}\). The previous two identities along with the Riesz-Thorin interpolation theorem give Hausdorff-Young's inequality, which gives a more generalized result.

We have
\begin{theorem}[Riesz-Thorin Interpolation Theorem]
Let \(T\) be a linear operator that maps the sum set \(L^{p0}+L^{p1}\) to the sum set \(L^{q0}+L^{q1}\), \(1\leqslant p_0,p_1,q_0,q_1\leqslant\infty\) such that \(p_0\neq p_1\) and \(q_0\neq q_1\). If \(\|T(f)\|_{L^{q_0}}\leqslant M_0\|f\|_{L^{p_0}}\) for all \(f\in L^{p_0}\) and \(\|T(g)\|_{L^{q_1}}\leqslant M_1\|g\|_{L^{p_1}}\) for all \(g\in L^{q_0}\), let \(\frac{1}{p}:=\frac{1-\theta}{p_0}+\frac{\theta}{p_1}\) and \(\frac{1}{q}:=\frac{1-\theta}{q_0}+\frac{\theta}{q_1}\), \(0<\theta<1\), 

then \(\|T(h)\|_{L^{p}}\leqslant M_0^{1-\theta}M_1^{\theta}\|h\|_{L^{q}}\) for all \(h\in L_{p}\).[1]
\end{theorem}

\begin{theorem}[Hausdorff-Young inequality]
\[\|\widehat{f}\|_{L^{q} (\mathbb{R}^n)}\leqslant\|f\|_{L^p(\mathbb{R}^n)}\] for all \(1\leqslant p\leqslant2\), \(f\in L^1\cap L^2\), and \(q=p'\).[2]
\end{theorem}

The Hausdorff-Young inequality gives the range of \(p,q\) where the \(L^q\) norm on the whole Fourier domain is less than the \(L^{p}\) norm on the whole spacial domain. This type of problem can be studied on smaller subspaces, such as \(\mathcal{M}^\psi\), the \(d\)-dimensional sub-manifold of \(\mathbb{R}^n\). To derive the Fourier transform of functions on \(\mathcal{M}^\psi\), the corresponding measure has to be defined first.

\begin{definition}[Pullback measure of \(\mathcal{M}^\psi\) from \(\mathbb{R}^d\)]
For the measure spaces \((\mathcal{M}^\psi,\mathcal{B}')\) and \((\mathbb{R}^d,\mathcal{B})\), where \(\mathcal{M}^\psi:=\{(\xi,\psi(\xi)):\xi\in U\subset\mathbb{R}^d\}\), \(\mathcal{B}\) is the Borel \(\sigma\)-algebra of \(\mathbb{R}^n\) and \(\mathcal{B}':=\{A\cap\mathcal{M}^\psi:A\in\mathcal{B}\}\). Let \(P:\mathcal{M}^\psi\to U\) be the canonical projection in the sense that \(P((\xi,\psi(\xi)))=\xi\in U\). Here \(P\) is an invertible map as it is a bijection. \\Define the pullback measure \(d\sigma^\psi\) of the Lebesgue measure \(d\mu\) under \(P\) as \[d\sigma^\psi(A):=d\mu(P(A)),\,\,\,\,(A\in\mathcal{B}').\]
\end{definition}

The Fourier restriction seeks to find the full range of \(p, q\) such that the inequality \(\|\widehat{f}\|_{L^q(\mathcal{M}^\psi,d\sigma^\psi)}\lesssim\|f\|_{L^p(\mathbb{R}^n)}\) holds, for all Schwartz functions \(f\in \mathcal{S}(\mathbb{R}^n)\).

\begin{definition}[Schwartz Space]
Define the Schwartz space \[\mathcal{S}(\mathbb{R}^n):=\{f\in C^\infty(\mathbb{R}^n): \forall\alpha,\beta\in \mathbb{N}^n, \|f\|_{\alpha,\beta}<\infty\},\] where \[\|f\|_{\alpha,\beta}:=\sup_{x\in\mathbb{R}^n}|x^\alpha(\partial^\beta f)(x)|.\]
\end{definition}

Note that \(\mathcal{S}(\mathbb{R}^n)\) is a subspace of \(L^p\) for all \(1\leqslant p\leqslant\infty\) so every \(L^p\) norm makes sense on Schwartz functions. The Schwartz space is also dense in \(L^p\) for all \(1\leqslant p<\infty\). Another property is that \(\widehat{f}\in\mathcal{S}(\mathbb{R}^n)\) for all \(f\in\mathcal{S}(\mathbb{R}^n)\). 

The convention of the restriction and extension in this dissertation is from [3]. To highlight the concept of restriction, given a d-dimensional manifold \(\mathcal{M}^\psi\), define the restriction operator on \(\mathcal{S}(\mathbb{R}^n)\) as: \[[\mathcal{R}^\psi f]:=(\int_{\mathbb{R}^n}f(x)e^{-2\pi i x\cdot \xi}dx)|_{\mathcal{M}^\psi},\]i.e. \emph{restrict} the Fourier transform on a manifold. We say that the expression \(R_{\mathcal{M}^\psi}(p\mapsto q)\) holds if \(\|\mathcal{R}^\psi f \|_{L^q(\mathcal{M}^\psi,d\sigma^\psi)}\lesssim\|f\|_{L^p(\mathbb{R}^n)}\) for all \(f\in\mathcal{S}(\mathbb{R}^n)\) with constant independent of \(f\).

An equivalent formulation of the restriction problem is to use the extension operator on \(L^1(\mathcal{M}^\psi)\), defined as: \[[\mathcal{E}^\psi f](x):=\int_{\mathcal{M}^\psi}f(\xi,\psi(\xi))e^{2\pi i(x'\cdot \xi+x^*\cdot \psi(\xi))}d\sigma^\psi (\xi,\psi(\xi)), \,x=(x',x^*)\in\mathbb{R}^d\times\mathbb{R}^{n-d},\]and the corresponding expression \(R^*_{\mathcal{M}^\psi}(q'\mapsto p')\) holds if \(\|\mathcal{E}^\psi f\|_{L^{p'}(\mathbb{R}^n)}\lesssim\|f\|_{L^{q'}(\mathcal{M}^\psi,d\sigma^\psi)}\) for all \(f\in L^{q'}(\mathcal{M}^\psi)\) with the constant independent of \(f\). Since \(U\) is bounded, \(\|\mathcal{E}^\psi f\|_{L^{p'}(\mathbb{R}^n)}\) is finite for \(f\in L^{q'}(\mathcal{M}^\psi)\).

The term \emph{extension} is intuitive as the extension operator is the formal adjoint of the restriction operator, in the sense that for \(f\in L^1(\mathcal{M}^\psi)\) and \(g\in\mathcal{S}(\mathbb{R}^n)\), by direct expansion and Fubini's theorem,\[\int_{\mathbb{R}^n}g(x)\overline{\mathcal{E}^\psi f(x)}dx=\int_{\mathcal{M}^\psi}\mathcal{R}^\psi g(\xi,\psi(\xi))\overline{ f(\xi,\psi(\xi))}d\sigma^\psi (\xi,\psi(\xi)).\]
An important consequence is the equivalence between restriction and extension, which provides an alternative way to solve the restriction problem.

\begin{lemma}
\(R_{\mathcal{M}^\psi}(p\mapsto q)\iff R^*_{\mathcal{M}^\psi}(q'\mapsto p')\).[3, p4]
\end{lemma}
\begin{proof}
The space \(L^{q'}\subset L^1\) for all \(q'\in[1,\infty]\). By the duality argument [4, p270], for all \(f\in S(\mathbb{R}^n)\), \[\|\mathcal{R}^\psi f\|_{L^q(\mathcal{M}^\psi,d\sigma^\psi)}=\sup_{\|g\|_{L^{q'}(\mathcal{M}^\psi)=1}}\int_{\mathcal{M}^\psi}\mathcal{R}^\psi f(\xi,\psi(\xi))\overline{g(\xi,\psi(\xi))}d\sigma^\psi(\xi,\psi(\xi)).\]
If \(R_{\mathcal{M}^\psi}(p\mapsto q)\) holds, then \(\|\mathcal{R}^\psi f \|_{L^q(\mathcal{M}^\psi,d\sigma^\psi)}\lesssim\|f\|_{L^p(\mathbb{R}^n)}\) all \(f\in S(\mathbb{R}^n)\). By normalization, the estimate has the form \(\|\mathcal{R}^\psi f \|_{L^q(\mathcal{M}^\psi,d\sigma^\psi)}\lesssim1\) for all \(f\in S(\mathbb{R}^n)\) and \(\|f\|_{L^p(\mathbb{R}^n)}=1\).

Define the norm of the restriction operator
\[\|\mathcal{R}^\psi\|_{p\to q}:=\sup_{f\in \mathcal{S}(\mathbb{R}^n):\,\|f\|_{L^p}=1,\|g\|_{L^{q'}(\mathcal{M}^\psi)=1}}|\int_{\mathcal{M}^\psi}\mathcal{R}^\psi f(\xi,\psi(\xi))\overline{g(\xi,\psi(\xi))}d\sigma^\psi(\xi,\psi(\xi))|.\]
Similarly, the extension operator has the form
\[\|\mathcal{E}^\psi\|_{q'\to p'}:=\sup_{f\in \mathcal{S}(\mathbb{R}^n):\,\|f\|_{L^p}=1,\|g\|_{L^{q'}(\mathcal{M}^\psi)=1}}|\int_{\mathbb{R}^n}\mathcal{E}^\psi g(x)\overline{f(x)}dx|.\]
Hence, \(R_{\mathcal{M}^\psi}(p\mapsto q)\) holds if and only if \(\|\mathcal{R}^\psi\|_{p\to q}\lesssim1\). For the same argument, \(R^*_{\mathcal{M}^\psi}(q'\mapsto p')\) holds if and only if \(\|\mathcal{E}^\psi\|_{q'\to p'}\lesssim1\). Since the extension operator is the formal adjoint of the restriction operator, then \(\|\mathcal{R}^\psi\|_{p\to q}=\|\mathcal{E}^\psi\|_{q'\to p'}\).

Altogether, \[R_{\mathcal{M}^\psi}(p\mapsto q)\iff\|\mathcal{R}^\psi\|_{p\to q}\lesssim1\iff\|\mathcal{E}^\psi\|_{q'\to p'}\lesssim1\iff R^*_{\mathcal{M}^\psi}(q'\mapsto p').\]
\end{proof}

There is also an alternative form of the extension operator using the Lebesgue measure on \(U\), in the sense that writing \(f(\xi,\psi(\xi))\) as \(f_\psi(\xi)\),

\[\int_{\mathcal{M}^\psi}f(\xi,\psi(\xi))\,\, d\sigma^\psi(\xi,\psi(\xi))=\int_Uf_\psi(\xi)\,\, d\xi,\]
and the extension operator also has the form
\[[\mathcal{E}^\psi f](x):=\int_Uf_\psi(\xi)e^{2\pi i(x'\cdot \xi+x^*\cdot \psi(\xi))}d\xi.\]
This form is used in the sections about bilinear and multilinear restriction, which works on functions defined on \(U\), instead of \(\mathcal{M}^\psi\).

\begin{RT}[Linear Restriction on \(\mathbb{R}^n\)]
For \(1\leqslant p,q\leqslant\infty\), \(R_{\mathbb{R}^n}(p\mapsto q)\) holds if and only if \(q=p'\) and \(1\leqslant p \leqslant2\).[5]
\end{RT}

Together with the Hausdorff-Young inequality, the restriction problem on \(\mathbb{R}^n\) is solved by further showing that if \(R_{\mathbb{R}^n}(p\mapsto q)\) holds, then \(p=q’\) and \(1\leqslant p \leqslant2\). 

The idea of the proof is from [5]. To show the first condition \(p=q’\), assume \(R_{\mathbb{R}^n}(p\mapsto q)\) holds, i.e. \(\|\widehat{f}\|_{L^q(\mathbb{R}^n)}=\|\mathcal{R}^\psi f \|_{L^q(\mathbb{R}^n,d\mu)}\lesssim\|f\|_{L^p(\mathbb{R}^n)}\) for all \(f\in\mathcal{S}(\mathbb{R}^n)\). 

Then for all \(\lambda > 0\), let \(f_\lambda:=f(\frac{x}{\lambda})\). Note that \(f_\lambda\) is also a Schwartz function, with norm \(\|f_\lambda\|_{L^p(\mathbb{R}^n)}=\lambda^{\frac{n}{p}}\|f\|_{L^p(\mathbb{R}^n)}\) and we have the inequality \(\|\widehat{f_\lambda}\|_{L^q(\mathbb{R}^n)}\lesssim\|f_\lambda\|_{L^p(\mathbb{R}^n)}\). 

By the scaling property of the Fourier transform, \(\widehat{f_\lambda}(\xi)=\lambda^n\widehat{f}(\lambda\xi)\), hence by change of variables, \(\|\widehat{f}_\lambda\|_{L^q(\mathbb{R}^n)}=\lambda^{n-\frac{n}{q}}\|\widehat{f}\|_{L^q(\mathbb{R}^n)}\). 

Altogether,  \(\|\widehat{f}\|_{L^q(\mathbb{R}^n)}=\lambda^{\frac{n}{q}-n}\|\widehat{f_\lambda}\|_{L^q(\mathbb{R}^n)}\lesssim\lambda^{\frac{n}{q}-n}\|\widehat{f}\|_{L^p(\mathbb{R}^n)}= \lambda^{\frac{n}{p}+\frac{n}{q}-n}\|f\|_{L^p(\mathbb{R}^n)}\) for all \(\lambda>0\). 

If the exponent \(\frac{n}{p}+\frac{n}{q}-n\neq0\), then \(\|f\|_{L^p(\mathbb{R}^n)}=0\) by either taking the limit of \(\lambda\) to \(0\) or \(\infty\), which is not true in general. Hence  \(\frac{n}{p}+\frac{n}{q}-n=0\), which implies \(\frac{1}{p}+\frac{1}{p'}=1\).

Given the Hausdorff-Young inequality,  showing the necessity of the condition \(1\leqslant p \leqslant2\) is equivalent to showing that \(R_{\mathbb{R}^n}(p\mapsto q)\) fails for any \(p>2\). Rather than using deterministic methods, Khintchine's inequality provides a rapid way to provide a counterexample that the chosen distribution of Schwartz function constructs a contradiction.

\begin{lemma}[Khintchine's inequality for functions]
For \(0<p<\infty\), let \(\{\varepsilon_j\}^m_{j=1}\) be a sequence of independent random variables that takes values \(\{-1,+1\}\) with equal probability of \(\frac{1}{2}\), then for any \(\{f_j\}^m_{j=1}\subset L^p(X,\mu)\),[5]

\[(\mathbb{E}\|\sum_{j=1}^{m}\varepsilon_jf_j\|^p_{L^p(X,\mu)})^\frac{1}{p}\sim_p\|(\sum_{j=1}^{m}|f_j|^2)^\frac{1}{2}\|_{L^p(X,\mu)}.\]
\end{lemma}
Let \(\varphi\) be a mollified indicator function, also known as a bump function, supported on the open unit ball \(B_1(0)\). Note that a bump function is also a Schwartz function as \(\varphi\in C^\infty(\mathbb{R}^n)\) by construction, and by the smoothness of \(\varphi\),\[\|\varphi\|_{\alpha,\beta}=\sup_{x\in\mathbb{R}^n}|x^\alpha(\partial^\beta\varphi)(x)|\leqslant\|x^\alpha\|_{L^\infty(B_1(0))}\|\partial^\beta f(x)\|_{L^\infty(B_1(0))}<\infty.\]

Let \(f(x)=\sum_j^m\varepsilon_j(\varphi(x-x_j))\), where \(\{\varepsilon_j\}_{j=1}^m\) is a sequence of independent random variables as in the assumption of Khintchine's inequality, and \(\{x_j\}_{j=1}^m\) is a sequence of points in \(\mathbb{R}^n\) such that \(\|x_j-x_k\|\geqslant2\) for all \(1\leqslant j<k\leqslant m\). 

Here, \(f\) can be understood as a random variable taking values which are sums of bump functions. By linearity, when choosing a particular combination of the sequence \(\{\varepsilon_j\}_{j=1}^m\), \(f\) is a sum of bump functions with disjoint support, which is also a Schwartz function.

By the disjointness of \(\{\varphi(x-x_j))\}_{j=1}^m\),
\begin{align*}
\|f\|_{L^p(\mathbb{R}^n)}&=(\int_{\mathbb{R}^n}|\sum_{j=1}^m\varepsilon_j(\varphi(x-x_j))|^pdx)^\frac{1}{p}, &&\\
&=(\sum_{j=1}^m\int_{\mathbb{R}^n}|(\varphi(x-x_j))|^pdx)^\frac{1}{p},&&\\
&\sim_{p,\varphi}(\sum_{j=1}^mm)^\frac{1}{p}=m^\frac{1}{p}.&&
\end{align*}
If the estimate \(R_{\mathbb{R}^n}(p\mapsto p')\) holds, then deterministically, \(\|\widehat{f}\|_{L^{p'}(\mathbb{R}^n)}\lesssim\|f\|_{L^p(\mathbb{R}^n)}\sim_{p,\varphi} m^\frac{1}{p}\).

One the other side, utilizing the change of variables, Khintchine's inequality says that
\begin{align*}
(\mathbb{E}\|\widehat{f}\|_{L^{p'}(\mathbb{R}^n)}^{p'})^\frac{1}{p'}&=(\mathbb{E}\|\int_{\mathbb{R}^n}\sum_{j=1}^m\varepsilon_j(\varphi(x-x_j))e^{-2\pi ix\cdot \xi}dx\|_{L^{p'}(\mathbb{R}^n)}^{p'})^\frac{1}{p'},&&\\
&=(\mathbb{E}\|\sum_{j=1}^m\varepsilon_je^{2\pi ix_j\cdot\xi}\widehat{\varphi}(\xi)\|_{L^{p'}(\mathbb{R}^n)}^{p'})^\frac{1}{p'},&&\\
&\sim_{n,p}\|(\sum_{j=1}^m|e^{2\pi ix_j\cdot\xi}\widehat{\varphi}(\xi)|^2)^\frac{1}{2}\|_{L^{p'}(\mathbb{R}^n)}^{p'}\sim_{n,p,\varphi} m^\frac{1}{2}.&&
\end{align*}
Then the above two estimates gives \(m^\frac{1}{2}\lesssim_{n,p,\varphi} m^\frac{1}{p}\) for arbitrary \(m\), and the estimate is independent of \(m\). With this estimate's complete form \(m^\frac{1}{2}\leqslant C(n,p,\varphi) m^\frac{1}{p}\), for every \(n,\varphi\), and assuming \(p>2\), the inequality fails for all \(m>C(n,p,\varphi)^{\frac{1}{p}-\frac{1}{2}}\).

Hence, by the previous arguments, for \(1\leqslant p,q\leqslant\infty\), \(R_{\mathbb{R}^n}(p\mapsto q)\) holds if and only if \(p=q’\) and \(1\leqslant p \leqslant2\). 

Restriction on \(\mathbb{R}^n\) is relatively simple, as the statement has been completely proved and is also strict. However, the restriction problem on manifolds is far more complex and they remained as conjectures on arbitrary dimensions. The more commonly studied cases are the restriction problems on the truncated paraboloid \(\mathbb{P}^{n-1}\)and the sphere \(\mathbb{S}^{n-1}\), where \[\mathbb{P}^{n-1}:=\{(\xi_1,\cdots,\xi_{n-1},\sum_{n-1}^{j=1}\xi^2_j):|\xi_j|<1\,\forall\,j\leqslant n-1\},\]
\[\mathbb{S}^{n-1}:=\{(\xi_1,\cdots,\xi_{n}):\sum_{n}^{j=1}\xi^2_j=1\}.\]
\(\mathbb{S}^{n-1}\) is not a manifold of the form \(\mathcal{M}^\psi:=\{(\xi,\psi(\xi)):\xi\in U\}\). However, the restriction on \(\mathbb{S}^{n-1}\) is the same as the restriction on the open hemisphere \(\mathbb{S}^{n-1}_\pm:=\{(\xi,\pm\sqrt{1-|\xi|^2}):|\xi|^2<1\}\), which is of the form \(\mathcal{M}^\psi\).

\begin{conjecture}[Linear Restriction Conjecture for \(\mathbb{P}^{n-1}\) and \(\mathbb{S}^{n-1}\)]
For \(1\leqslant p,q\leqslant\infty\), \(R_{\mathbb{P}^{n-1}}(p\mapsto q)\) and \(R_{\mathbb{S}^{n-1}}(p\mapsto q)\) hold if and only if \(\frac{n-1}{q}\geqslant\frac{n+1}{p'}\) and \(p'>\frac{2n}{n-1}\).
\end{conjecture}

\begin{figure}[hbt!]
    \centering
    \includegraphics[width=1\linewidth]{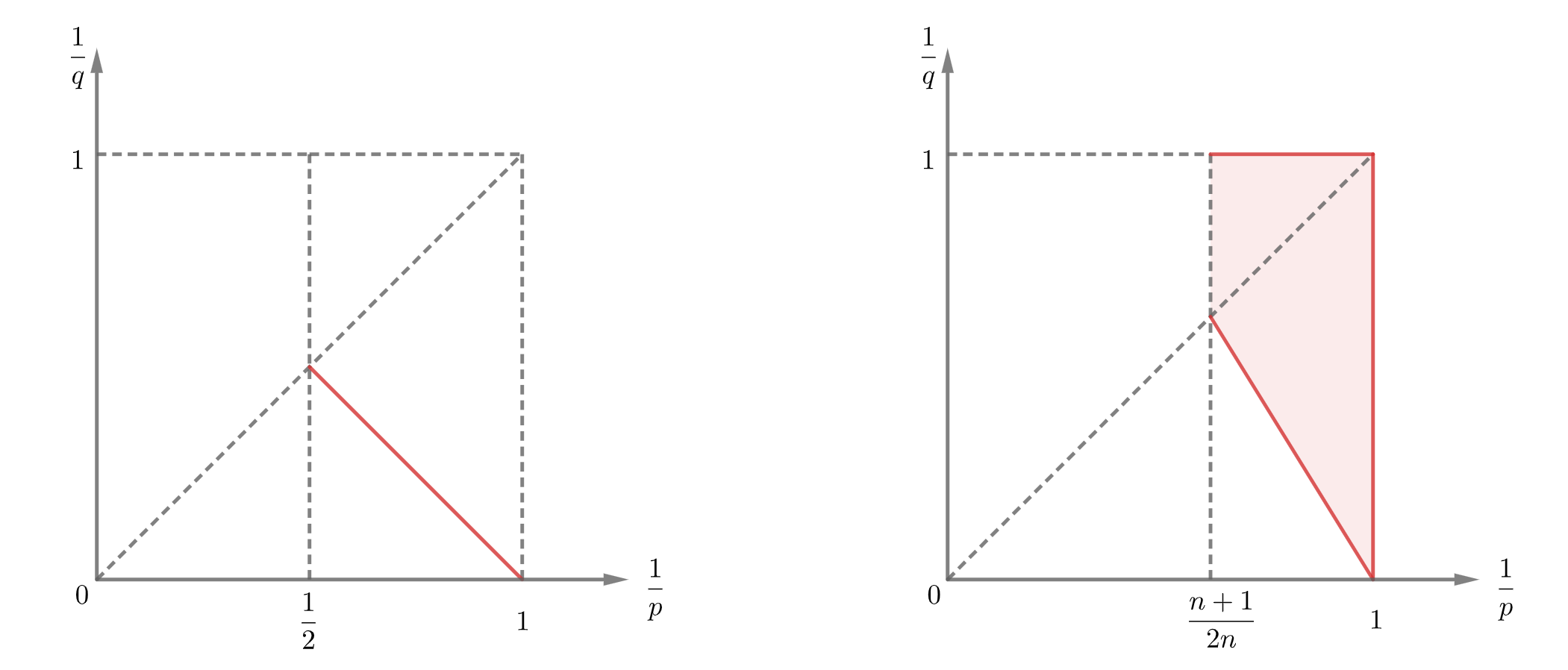}
    \caption{Left: Restriction range for \(\mathbb{R}^n\).       
    Right: Restriction range for \(\mathbb{P}^{n-1}\) and \(\mathbb{S}^{n-1}\).}
    \label{fig:enter-label}
\end{figure}

It can be clearly seen from the graph that the restriction range on specific manifolds is larger than the whole space.

There is also an equivalent form of Conjecture 1.

\begin{conjecture}[Linear Restriction Conjecture for \(\mathbb{P}^{n-1}\) and \(\mathbb{S}^{n-1}\) variant]
For \(1\leqslant p\leqslant\infty\), \(R^*_{\mathbb{P}^{n-1}}(\infty\mapsto p')\) and \(R^*_{\mathbb{S}^{n-1}}(\infty\mapsto p')\) hold if and only if \(p'>\frac{2n}{n-1}\).[4, p270]
\end{conjecture}

By testing the estimates using the Knapp example on cubes, which is applicable to arbitrary hypersurfaces [3, p7-p8], it is shown that the condition \(\frac{n-1}{q}\geqslant\frac{n+1}{p'}\) in Conjecture 1 is necessary to ensure the extension estimate holds. The necessity of the condition \(p'>\frac{2n}{n-1}\) can be shown by analyzing the Fourier transform of full surface measure, where the extension estimates diverge when \(p'\leqslant\frac{2n}{n-1}\).[3, p9] Though Conjecture 1 remains unsolved for arbitrary dimensions, there are two types of arguments when attempting to solve the conjecture, both of which will be presented below. 

The first type of method is to solve the restriction problem sharply on a fixed dimension. In the 2-dimensional case, the restriction conjecture is fully verified for curves with nonzero curvatures.

\begin{RT}[Linear Restriction on 2-dimensional curves with non-zero curvatures]
For \(1\leqslant p,q\leqslant\infty\), \(\psi:[-1,1]\to\mathbb{R}\) with non-zero curvature, then  \(R^*_{\mathcal{M}^\psi}(q'\mapsto p')\) holds if and only if \(p'>4\) and \(\frac{3}{p'}+\frac{1}{q'}\leqslant1\). [3, p10-p11]
\end{RT}

Using the previous notation of manifolds, let \(\psi:[-1,1]\to \mathbb{R}\) a \(C^2\) function such that \(\inf_{|\xi|\leqslant1}|\psi''(\xi)|\geqslant v>0\). Then the extension operator has the form: \[\mathcal{E}^\psi f(x_1,x_2):=(\int_{\mathcal{M}^\psi}f(\xi)e^{2\pi i (x_1\xi+x_2\psi(\xi))}d\sigma^\psi(\xi,\psi(\xi))).\]
Since the restriction outside of the conjectured range is false, it suffices to show the extension estimate holds on the endpoint result \(\frac{3}{p'}+\frac{1}{q'}=1\).

Let \(2r'=p'\), then the endpoint result is equivalent to \(\frac{2r}{3-r}=q'\), and \[\|\mathcal{E}^\psi f\|^2_{L^{p'}(\mathbb{R}^2)}=\|(\mathcal{E}^\psi f)^2\|_{L^{r'}(\mathbb{R}^2)}.\]
The assumption \(\psi\in C^2\) and \(\inf_{|\xi|\leqslant1}|\psi''(\xi)|\geqslant v>0\) implies that \(\psi'\)  is either convex or concave in the whole domain. Hence by Taylor's expansion, \(|\psi'(t)-\psi'(s)|\geqslant v|t-s|\). 

By change of variables using \((t,s)=(T(\xi_1,\xi_2)\), where \((\xi_1,\xi_2)=(t+s,\psi(t)+\psi(s))\). Then T is invertible for \(t>s\) and \(\det[T'(\xi_1,\xi_2)]=(\psi'(t)-\psi'(s))^{-1}\). Hence
\begin{align}
(\mathcal{E}^\psi f)^2(x_1,x_2)&=2\iint_{t>s}f(t)f(s)e^{2\pi i(x_1(t+s)+x_2(\psi(t)+\psi(s)))}dtds, &&\\
&=\iint(f\otimes f)(T(\xi_1,\xi_2))|\det[T'(\xi_1,\xi_2)]|e^{2\pi i(x_1\xi_1+x_2\xi_2)}d\xi_1d\xi_2,&&
\end{align}
which is the inverse Fourier transform of \[F(\xi_1,\xi_2):=(f\otimes f)(T(\xi_1,\xi_2))|\det[T'(\xi_1,\xi_2)]|.\]
Here the tensor product of two functions is defined as \(f\otimes g:\mathbb{R}^2\to \mathbb{C}\), \(f\otimes g(x,y):=f(x)g(y)\).

By construction \(1\leqslant r\leqslant2\), the Hausdorff-Young inequality implies
\[\|(\mathcal{E}^\psi f)^2\|_{L^{r'}(\mathbb{R}^2)}=(\int |F(\xi)e^{-2\pi ix\cdot\xi}e^{4\pi ix\cdot\xi}d\xi|^{r'})^{\frac{1}{r'}}\leqslant\|\widehat{F}\|_{L^{r'}(\mathbb{R}^2)}\leqslant\|F\|_{L^{r}(\mathbb{R}^2)}.\]
H\"older's inequality says that
\begin{theorem}[H\"older's inequality]
For \(1\leqslant p,q\leqslant\infty\) and \(f,g\) measurable functions.
If \(\frac{1}{p}+\frac{1}{q}=1\), then \(\|fg\|_{L^1}\leqslant\|f\|_{L^p}\|g\|_{L^q}\).
\end{theorem}
By further change of variables and H\"older's inequality,
\begin{align*}
\|F\|^r_{L^r(\mathbb{R}^2)}&=\iint|f(t)f(s)|^r|\psi'(t)-\psi'(s)|^{1-r}dtds,&&\\
&\lesssim_v\iint|f(t)f(s)|^r|t-s|^{1-r}dtds,&&\\
&\leqslant\|f^r\|_{L^\frac{q'}{r}(\mathbb{R})}\||f|^r*|t|^{1-r}\|_{L^{(\frac{q'}{r})'}(\mathbb{R})}.&&
\end{align*}
The Hardy-Littlewood-Sobolev inequality[8] states that \[\||f|^r*|t|^{1-r}\|_{L^{(\frac{q'}{r})'}(\mathbb{R})}\lesssim\|f^r\|_{L^{\frac{q'}{r}}(\mathbb{R})}.\]
Hence the estimate\[\|\mathcal{E}^\psi f\|^2_{L^{p'}(\mathbb{R}^2)}=\|(\mathcal{E}^\psi f)^2\|_{L^{r'}(\mathbb{R}^2)}\lesssim\|F\|_{L^{r}(\mathbb{R}^2)}\lesssim_v\|f^r\|^\frac{2}{r}_{L^{\frac{q'}{r}}(\mathbb{R})},\]
which implies  \(\|\mathcal{E}^\psi f\|_{L^{p'}(\mathbb{R}^2)}\lesssim_v\|f\|_{L^{q'}([-1,1])}\) for each \(p'>4\) and \(\frac{3}{p'}+\frac{1}{q'}\leqslant1\). This is precisely the statement of Conjecture 1 when \(n=2\).

The proofs of higher-dimensional cases are significantly more difficult, and all attempts are focused on approximating the restriction range to the conjectured range. On other manifolds, notably the truncated cone, the restriction conjecture is verified in some higher dimensions.

The second type method is to solve the restriction problem by showing the restriction estimate holds on a subset of the conjectured range, on arbitrary dimensions.

\begin{RT}[Linear Restriction of the paraboloid for \(q=2\)]
For \(1\leqslant p,q\leqslant\infty\),  \(R_{\mathcal{M}^\psi}(p\mapsto q)\) holds when \(q=2\) and \(1\leqslant p\leqslant\frac{2(n+1)}{n+3}\). [3, p12-p14]
\end{RT}

The proof is called the Stein-Tomas argument, which proves the conjecture on the sphere and paraboloid when \(q=2\) and \(1\leqslant p<\frac{2(n+1)}{n+3}\). The endpoint result \(p=\frac{2(n+1)}{n+3}\) requires complex interpolation.[6] Note that the restriction range is identical to the conjectured range when \(q=2\).

\begin{definition}[Fourier transform of measure]
The Fourier transform of the measure \(d\sigma^\psi\) on a manifold \(\mathcal{M}^\psi\) is defined as  
\[\widehat{d\sigma^\psi}(\xi):=\int_{\mathcal{M}^\psi}e^{-2\pi ix\cdot\xi}d\sigma^\psi(x).\]
[7, p347]
\
\end{definition}
Without loss of generality, let \(x_0\) be a point on a 1-dimensional manifold \(\mathcal{M}^\psi\) such that \(\nabla\psi(x_0)=0\). This can be achieved by a rotation and a translation of the whole space \(\mathbb{R}^n\). Then the principal curvatures of \(\mathcal{M}^\psi\) at \(x_0\) are defined as the eigenvalues \(\{v_l\}^{n-1}_{l=1}\) of the \((n-1)\times(n-1)\) Hessian matrix \[\left[\frac{\partial^2\psi(x_0)}{\partial x_j\partial x_k}\right].\]
The Gaussian curvature of \(\mathcal{M}^\psi\) at \(x_0\) is defined as \(\prod^{n-1}_{l=1}v_l\).[7, p348-p350] If the Gaussian curvature is nonzero everywhere on \(\mathcal{M}^\psi\), \(\psi\in C^\infty\), then [7, p348-p350]\[|\widehat{d\sigma^\psi}(x)|\lesssim(1+|x|)^{-\frac{n-1}{2}},\]
where \(d\sigma^\psi\) is the pullback measure of \(d\mu\). Note that both \(\mathbb{S}^n\) and \(\mathbb{P}^n\) have nonzero Gaussian curvatures.

Let \(U\subset\mathbb{R}^{n-1}\) be the cube where the paraboloid \(\mathbb{P}^{n-1}\) is supported. Since for \(\mathcal{M}^\psi=\mathbb{P}^{n-1}\),  there is a continuous extension \(\widetilde{\psi}:2U\to \mathbb{R}\), which is the paraboloid on \(2U\), where \(2U\) is the dilate of U by a factor of 2 from the centre, and \(\widetilde{\psi}=\psi\) on \(U\).

Let \(w\in C^\infty(\mathbb{R}^n)\) be the smooth cutoff function such that \(\mathbbm{1}_U\leqslant w\leqslant\mathbbm{1}_{2U}\). More formally, \(w=\mathbbm{1}_V*\rho_\varepsilon\), where \(U\subsetneq V\subset2U\) and \(\rho_\varepsilon\) is a mollifier with small enough \(\varepsilon\) such that \(w=1\) on \(U\). Define the measure \(\widetilde{d\sigma^{\psi}}\) on \(\mathcal{M}^{\widetilde{\psi}}\) as the pullback measure of \(wd\mu\) on \(\mathbb{R}^{n-1}\).

As we will show later, the restriction estimate holds if \(\|\mathcal{R}^{\widetilde{\psi}}\|_{L^p(\mathbb{R}^n,dx)\to L^2(\mathcal{M}^{\widetilde{\psi}},\widetilde{d\sigma^\psi)}}<\infty\).

The corresponding adjoint of the restriction operator, which is the extension operator, has the form
\[\widetilde{\mathcal{R}^{\psi,*}}f(x)=\int_{\mathcal{M}^{\widetilde{\psi}}}f(\xi,\widetilde{\psi}(\xi))e^{-2\pi i(x'\cdot\xi+x^*\cdot\phi(\xi))}\widetilde{d\sigma^{\psi}}.\]
Since \(\widetilde{\mathcal{R}^{\psi,*}}\) is a map between Hilbert spaces, hence \(\|\widetilde{\mathcal{R}^{\psi,*}}\mathcal{R}^{\widetilde{\psi}}\|_{L^p(\mathbb{R}^n,dx)\to L^{p'}(\mathbb{R}^n,dx)}=\|\mathcal{R}^{\widetilde{\psi}}\|_{L^p(\mathbb{R}^n,dx)\to L^2(\mathcal{M}^{\widetilde{\psi}},\widetilde{d\sigma^\psi)}}^2\), which reformulates the problem into showing that \(\|\widetilde{\mathcal{R}^{\psi,*}}\mathcal{R}^{\widetilde{\psi}}\|_{L^p\to L^{p'}}<\infty\), and further equivalent to showing that \(\|\widetilde{\mathcal{R}^{\psi,*}}\mathcal{R}^{\widetilde{\psi}}f\|_{L^{p'}}\lesssim\|f\|_{L^p}\) for all \(f\in S(\mathbb{R}^n)\). 

The composite can be simplified as
\begin{align*}
\widetilde{\mathcal{R}^{\psi,*}}\mathcal{R}^{\widetilde{\psi}}f(x)&=\int_{\mathcal{M}^{\widetilde{\psi}}}\widehat{f}(\xi,\psi(\xi))e^{2\pi i(x'\cdot\xi+x^*\cdot\widetilde{\psi}(\xi))}\widetilde{d\sigma^{\psi}},&&\\
&=\int_{\mathcal{M}^{\widetilde{\psi}}}\int_{\mathbb{R}^n}f(y)e^{-2\pi i(y'\cdot\xi+y^*\cdot\widetilde{\psi}(\xi))}dye^{2\pi i(x'\cdot\xi+x^*\cdot\widetilde{\psi}(\xi))}\widetilde{d\sigma^{\psi}},&&\\
&=\int_{\mathbb{R}^n}f(y)\int_{\mathcal{M}^{\widetilde{\psi}}}e^{2\pi i((x'-y')\cdot\xi+(x^*-y^*)\cdot\widetilde{\psi}(\xi))}\widetilde{d\sigma^{\psi}}dy,&&\\
&=f*\widehat{d\lambda}(x),
\end{align*}
where \(d\lambda(v):=\widetilde{d\sigma^{\psi}}(-v)\). Hence it suffices to show that \(\|f*\widehat{d\lambda}\|_{L^{p'}}\lesssim\|f\|_{L^p}\).

The function \(\widehat{d\lambda}\) is defined on the whole space, so it can be decomposed into a sum of localized functions. When the piecewise estimates are proved for the desired range of \(p\) using interpolation as in the proof of the Hausdorff Young inequality, by the triangle inequality, the restriction estimate holds in the same range of \(p\).

Decompose the constant \(1\) by 
\[1=\varphi(x)+\sum_{j=1}^\infty\varphi_j(x),\]
where \(\varphi_0\in C_c^\infty(\mathbb{R}^n)\) is a bump function supported on \(B(0,1)\) such that \(\varphi_0(x)=1\) on \(B(0,\frac{1}{2})\), and \(\varphi_j(x):=\varphi_0(\frac{x}{2^j})-\varphi_0(\frac{x}{2^{j-1}})\) for all \(j\geqslant1\). So
\[\widehat{d\lambda}=\sum_{j=0}^\infty\varphi_j(x)\widehat{d\lambda}.\]
Further denote \(K_j:=\varphi_j\widehat{d\lambda}\) for all \(j\geqslant0\).

Since \(\widetilde{\psi}\) has nonzero Gaussian curvature, it satisfies \(|\widehat{d\lambda}(x)|=|\widehat{\widetilde{d\sigma^\psi}}(x)|\lesssim(1+|x|)^{-\frac{n-1}{2}}\), which implies
\[\|K_j\|_{L^\infty(\mathbb{R}^n)}\lesssim\|\widehat{d\lambda}\|_{L^\infty(\supp(\varphi_j))}\lesssim2^{-\frac{j(n-1)}{2}}.\]
Young's convolution inequality says that[9, p205-p206]
\begin{theorem}[Young's convolution inequality]
For \(1\leqslant p,q,r\leqslant\infty\) and \(f,g\) are measurable functions with \(f\in L^p(\mathbb{R}^n)\), \(g\in L^q(\mathbb{R}^n)\).

If \(\frac{1}{p}+\frac{1}{q}=\frac{1}{r}+1\), then \(\|f*g\|_{L^r}\leqslant\|f\|_{L^p}\|g\|_{L^q}\).
\end{theorem}
Hence let \(p=1,q=\infty,r=\infty\), Young's inequality shows that\[\|f*K_j\|_{L^\infty(\mathbb{R}^n)}\leqslant\|K_j\|_{L^\infty(\mathbb{R}^n)}\|f\|_{L^1(\mathbb{R}^n)}\lesssim2^{-\frac{j(n-1)}{2}}\|f\|_{L^1(\mathbb{R}^n)}.\]

Through some technical calculations, 
\(\|\widehat{K_j}\|_{L^\infty(\mathbb{R}^n)}\lesssim2^j\), with constants that depend only on \(\varphi\) and \(n\). [10] Hence by Plancherel's identity and the Fourier transform of convolutions, 
\[\|f*K_j\|_{L^2(\mathbb{R}^n)}=\|\widehat{f*K_j}\|_{L^2(\mathbb{R}^n)}=\|\widehat{f}\widehat{K_j}\|_{L^2(\mathbb{R}^n)}\lesssim2^j\|\widehat{f}\|_{L^2(\mathbb{R}^n)}=2^j\|f\|_{L^2(\mathbb{R}^n)}.\]
By the Riesz-Thorin Interpolation Theorem, let \(T\) be the Fourier transform, \(p_0=1\), \(p_1=2\), \(q_0=\infty\), \(q_1=2\) as previously. The \(L^\infty\) and \(L^2\) bounds of \(f*K_j\) give
\begin{align*}
\|f*K_j\|_{L^{p'}(\mathbb{R}^n)}&\lesssim2^{\frac{2j}{p'}}2^{-\frac{j(n-1)}{2}(1-\frac{2}{p'})}\|f\|_{L^p(\mathbb{R}^n)}&&\\
&=2^{j(\frac{n+1}{p'}-\frac{n-1}{2})}\|f\|_{L^p(\mathbb{R}^n)}
\end{align*}
for all \(j\geqslant0\). To ensure \(\sum_{j=0}^\infty\|f*K_j\|_{L^{p'}(\mathbb{R}^n)}\) converges, the exponent \(j(\frac{n+1}{p'}-\frac{n-1}{2})\) has to be negative to make the power series \(\sum_{j=0}^\infty2^{j(\frac{n+1}{p'}-\frac{n-1}{2})}\|f\|_{L^p(\mathbb{R}^n)}\) converge, which requires that \(p<\frac{2(n+1)}{n+3}\).

By the triangle inequality, 
\begin{align*}
\|f*\widehat{d\lambda}\|_{L^{p'}(\mathbb{R}^n)}&=\|\sum_{j=0}^\infty f*K_j\|_{L^{p'}(\mathbb{R}^n)},&&\\
&\leqslant\sum_{j=0}^\infty\| f*K_j\|_{L^{p'}(\mathbb{R}^n)},&&\\
&\lesssim\sum_{j=0}^\infty2^{j(\frac{n+1}{p'}-\frac{n-1}{2})}\|f\|_{L^p(\mathbb{R}^n)}\lesssim\|f\|_{L^p(\mathbb{R}^n),}
\end{align*}
which proves \(\|\mathcal{R}^{\widetilde{\psi}}\|_{L^p(\mathbb{R}^n,dx)\to L^2(\mathcal{M}^{\widetilde{\psi}},\widetilde{d\sigma^\psi)}}<\infty\) for \(p<\frac{2(n+1)}{n+3}\) and \(q=2\). Since by definition
\begin{align*}
&\|\mathcal{R}^{\widetilde{\psi}}\|_{L^p(\mathbb{R}^n,dx)\to L^2(\mathcal{M}^{\widetilde{\psi}},\widetilde{d\sigma^\psi)}}&&\\
&=\sup_{f\in \mathcal{S}(\mathbb{R}^n):\,\|f\|_{L^p}=1,\|g\|_{L^2(\mathcal{M}^{\widetilde{\psi}},\widetilde{d\sigma^\psi})=1}}|\int_{\mathcal{M}^{\widetilde{\psi}}}\mathcal{R}^\psi f(\xi,\psi(\xi))\overline{g(\xi,\psi(\xi))}\widetilde{d\sigma^\psi}\xi|,
\end{align*}
taking \(h\) with \(\|h\|_{L^2(\mathcal{M}^\psi,d\sigma^\psi)=1}\), then
\begin{align*}
&\|\mathcal{R}^\psi\|_{L^p(\mathbb{R}^n,dx)\to L^2(\mathcal{M}^\psi,d\sigma^\psi)}&&\\
&=\sup_{f\in \mathcal{S}(\mathbb{R}^n):\,\|f\|_{L^p}=1,\|h\|_{L^2(\mathcal{M}^\psi,d\sigma^\psi)=1}}|\int_{\mathcal{M}^\psi}\mathcal{R}^\psi f(\xi,\psi(\xi))\overline{h(\xi,\psi(\xi))}d\sigma^\psi\xi|,&&\\
&=\sup_{f\in \mathcal{S}(\mathbb{R}^n):\,\|f\|_{L^p}=1,\|h\|_{L^2(\mathcal{M}^\psi,d\sigma^\psi)=1}}|\int_{\mathcal{M}^{\widetilde{\psi}}}\mathcal{R}^\psi f(\xi,\psi(\xi))\overline{h(\xi,\psi(\xi))}\widetilde{d\sigma^\psi}\xi|,&&\\
&\leqslant\sup_{f\in \mathcal{S}(\mathbb{R}^n):\,\|f\|_{L^p}=1,\|g\|_{L^2(\mathcal{M}^{\widetilde{\psi}},\widetilde{d\sigma^\psi})=1}}|\int_{\mathcal{M}^{\widetilde{\psi}}}\mathcal{R}^\psi f(\xi,\psi(\xi))\overline{g(\xi,\psi(\xi))}\widetilde{d\sigma^\psi}\xi|,&&\\
&=\|\mathcal{R}^{\widetilde{\psi}}\|_{L^p(\mathbb{R}^n,dx)\to L^2(\mathcal{M}^{\widetilde{\psi}},\widetilde{d\sigma^\psi)}}<\infty.
\end{align*}
Which proves the restriction conjecture for \(q=2\).

The same result from the Stein Tomas argument is shown for the paraboloid \(\mathbb{P}^{n-1}\) for more general quadratic surfaces By Strichartz.[11] Also, the proof of restriction on the sphere is almost the same as the proof on the paraboloid.[3, p18]

\section{Bilinear Restriction}
Aside of attempting to prove the linear restriction conjecture via the two types of argument which are discussed in the previous section, the study of the linear restriction problem has motivated a class of more generalized problems called multilinear restriction problems. While still trying to find the full restriction range, the term multilinear refers to the problem involving multiple functions, instead of a single function in the linear case.

Let \(1\leqslant k\leqslant n\), for each \(1\leqslant j\leqslant k\), let \(\mathcal{M}^{\psi_j}:=\{(\xi,\psi_j(\xi)):\xi\in U_j\}\) be a sequence of smooth manifolds, the k-linear restriction problem seeks to find the range of \(p,q\) such that for all \(f_j\in L^p(\mathcal{M}^{\psi_j})\), the estimate
\[\||\prod_{j=1}^k\mathcal{E}^{\psi_j}f_j|^\frac{1}{k}\|_{L^p(\mathbb{R}^n,dx)}\lesssim(\prod_{j=1}^k\|f_j\|_{L^q(\mathcal{M}^{\psi_j},d\sigma^\psi)})^\frac{1}{k}\]
holds. When \(k=1\), the problem is identical to the linear case. The \(k=2\) case is called the bilinear restriction problem. 

The significance of bilinear restriction and multilinear restriction lies not only in their connections to the multilinear Kakeya estimates, they can also be exploited to prove linear restriction theorems, which will be discussed later for the bilinear restriction estimate.

\subsection{\(L^2\) Bilinear Restriction Theory}
Similar to the result from the Stein-Tomas restriction theorem in the previous section, the bilinear setting has an analogous problem as the linear case, which seeks to find the range of \(p\) such that on the paraboloid, the bilinear extension estimate holds for \(q=2\).

Having fixed the manifold for this section as the paraboloid, denote
\[\mathcal{E}_{\Omega}f:=\int_\Omega f(\xi)e^{2\pi i(x'\cdot\xi+x^*|\xi|^2)}d\xi,\,x=(x',x^*)\in\,\mathbb{R}^{n-1}\times\mathbb{R}.\]
Note that this formulation is the alternative form of the extension operator discussed in the previous section, and this convention will be used throughout the rest of the dissertation.
\begin{RT}[Bilinear Restriction of the paraboloid for \(q=2\)]
For \(n\geqslant2\), let \(\Omega_1, \Omega_2\) be two cubes in \([-1,1]^{n-1}\) with \(d(\Omega_1,\Omega_2):=\inf_{x_1\in\Omega_1,x_2\in\Omega_2}\|x_1-x_2\|>0\), then for all \(f:\Omega_1\cup\Omega_2\to\mathbb{C}\),
\[\||\mathcal{E}_{\Omega_1}f\mathcal{E}_{\Omega_2}f|^\frac{1}{2}\|_{L^p(\mathbb{R}^n)}\lesssim(\|f\|_{L^2(\Omega_1)}\|f\|_{L^2(\Omega_2)})^\frac{1}{2}\]
holds for all \(p>\frac{2(n+2)}{n}\).
[3, p35]
\end{RT}

The 2-dimensional case also holds for the endpoint \(p=4\), which will be proved in this section. By multilinear interpolation and H\"older's inequality, it suffices to prove the endpoint result.[12]

Following [3, p37-p44], there are two major components to prove Restriction Theorem 4, the first being the reverse square function estimate, which requires diagonal behavior for parabolic caps in its proof.

Let \(I\) be an interval, \(\sigma\) be a constant and denote the parabolic cap \(\mathcal{N}_I(\sigma)\) as \(\mathcal{N}_I(\sigma):=\{(\xi,\xi^2+t):\xi\in I,|t|\leqslant\sigma\}\). For two sets \(\mathcal{N}_I(\sigma)\), \(\mathcal{N}_J(\sigma)\), define their difference as \(\mathcal{N}_I(\sigma)-\mathcal{N}_J(\sigma):=\{u-v:u\in\mathcal{N}_I(\sigma),v\in\mathcal{N}_J(\sigma)\}\). The diagonal behavior of parabolic caps says that the difference between two pairs of caps are disjoint if they are far apart.

\begin{theorem}[Diagonal behavior for parabolic caps]
(i) For each interval \(I_1,I_2,I_2'\subset[-1,1]\) with length \(\delta\ll1\), \(d(I_2, I'_2)\geqslant N\delta\) with \(N\) a large number, then
\[(\mathcal{N}_{\mathcal{I}_1}(\delta^2)-\mathcal{N}_{\mathcal{I}_1}(\delta^2))\cap(\mathcal{N}_{\mathcal{I}_2}(\delta^2)-\mathcal{N}_{\mathcal{I}_2'}(\delta^2))=\emptyset.\]
(ii) For each interval \(I_1,I_1',I_2,I_2'\subset[-1,1]\) with length \(\delta\ll1\), \(d(I_1, I'_1)\), \(d(I_2, I'_2)\geqslant N\delta\) with \(N\) is a large number, and \(\max(d(I_1, I'_1),d(I_2, I'_2))\geqslant N\delta\) then
\[(\mathcal{N}_{\mathcal{I}_1}(\delta^2)-\mathcal{N}_{\mathcal{I}_1'}(\delta^2))\cap(\mathcal{N}_{\mathcal{I}_2}(\delta^2)-\mathcal{N}_{\mathcal{I}_2'}(\delta^2))=\emptyset.\]
 with inequality constant independent of \(\delta\). [3, p38-p39]
\end{theorem}
\begin{proof}
Note that each \(\mathcal{N}_I(\delta^2)\) is contained in a ball of radius \(2\delta\). To show this, without loss of generality, let \(I=[a,a+\delta]\), \(a\in[-\delta,1-\delta]\). Then the four ``vertices" of the cap are \((a,a^2-\delta^2)\), \((a,a^2+\delta^2)\), \(((a+\delta)^2-\delta^2)\), \(((a+\delta)^2+\delta^2)\). 

To construct a parallelogram that contains the cap, connect \((a,a^2+\delta^2)\) and \(((a+\delta)^2+\delta^2)\) linearly gives a line \(l_1:y=(2a+\delta)\xi+a\delta+\delta^2-a^2\). Similarly, the tangent line of \(y=\xi^2-\delta^2\) that is parallel to \(l_1\) has the form \(l_2:y=(2a+\delta)\xi-a^2-a\delta-\frac{5}{4}\delta^2\). 

Then \(\mathcal{N}_I(\delta^2)\) is contained in the parallelogram \(\{(\xi,y):\xi\in[a,a+\delta],l_2\leqslant y\leqslant l_1\}\), with vertices \((a,a^2+\delta^2)\), \(((a+\delta)^2+\delta^2)\), \((a,a^2-\frac{5}{4}\delta^2)\), \((a+\delta,a^2+2a\delta-\frac{1}{4}\delta^2)\). 

Since the longest line segment in the parallelogram is the line connecting \((a,a^2-\frac{5}{4}\delta^2)\) and \(((a+\delta)^2+\delta^2)\) , which, by the Pythagorean theorem, has length \(\sqrt{4a^2\delta^2+13a\delta^3+\frac{169}{16}\delta^4+\delta^2}\). Since \(\delta\ll1\), taking \(a=1\) shows that the parallelogram is contained in a ball of radius \(2\delta\).

Since \(\mathcal{N}_I(\delta^2)\) is contained in \(B(c,2\delta)\), where \(c=(c_1,c_1^2)\) and \(c_1\) is the midpoint of \(I_1\), the set \(\mathcal{N}_{I_1}(\delta^2)-\mathcal{N}_{I_1}(\delta^2)\) is contained in the set \(B(c,2\delta)-B(c,2\delta)=\{c+u:|u-c|\leqslant2\delta\}-\{c+v:|v-c|\leqslant2\delta\}=\{u-v:|u-v|\leqslant4\delta\}=B(0,4\delta)\). Similarly, \(\mathcal{N}_{I_2}(\delta^2)-\mathcal{N}_{I_2'}(\delta^2)=B(c',4\delta)\), where \(c'=((c_2,c_2^2)-(c_2',(c_2'))\) and \(c_2\) and \(c_2'\) are the midpoints of \(I_2\) and \(I_2'\) respectively. Since \(|c'|\geqslant N\delta>8\delta\), conclusion (i) holds.

For (ii), assume \(\mathcal{N}_{\mathcal{I}_1}(\delta^2)-\mathcal{N}_{\mathcal{I}_1'}(\delta^2)\cap\mathcal{N}_{\mathcal{I}_2}(\delta^2)-\mathcal{N}_{\mathcal{I}_2'}(\delta^2)\neq\emptyset\) and \(d(I_1, I'_1)\), \(d(I_2, I'_2)\geqslant N\delta\), then by assumption, there exists \(\xi_1\in I_1\), \(\xi_1'\in I_1'\), \(\xi_2\in I_2\), \(\xi_2\in I_2'\) such that 
\[v:=\xi_1-\xi_1'=\xi_2-\xi_2', |v|\geqslant N\delta,\]\[ w_1:=\xi_1^2-(\xi_1')^2, w_2:=\xi_2^2-(\xi_2')^2,|w_1-w_2|\leqslant4\delta^2\]
by the previous two constructions,
\[\delta\geqslant|\frac{w_1-w_2}{v}|=|\frac{\xi_1^2-(\xi_1')^2}{\xi_1-\xi_1'}-\frac{\xi_2^2-(\xi_2')^2}{\xi_2-\xi_2'}|=|\xi_1-\xi_2+\xi_1'-\xi_2'|\]
Since \(\xi_1-\xi_2=\xi_1'-\xi_2'\), the above's inequality implies that \(|\xi_1-\xi_2|\), \(|\xi_1'-\xi_2'|\leqslant\delta\), and hence \( d(I_1, I'_1),d(I_2, I'_2)\leqslant N\delta\), and (ii) holds.
\end{proof}
With the diagonal behavior of parabolic caps, the proof of the reverse square function inequality is greatly reduced into two simple cases.

For two intervals \(J_1,J_2\subset[-1,1]\), let \(\mathcal{P}_j\) be a partition of \(J_j\) into intervals \(I_j\) with length \(\delta\) for each \(j=\{1,2\}\). For \(f:\mathbb{R}^2\to \mathbb{C}\) a function and an interval \(I\subset\mathbb{R}\), denote \(f_I\) be the Fourier projection of F onto the parabolic cap \(\mathcal{N}_I(\delta^2)\), where
\[\widehat{f_I}=\widehat{f}|_{\mathcal{N}_I(\delta^2)}.\]

\begin{theorem}[Reverse square function inequality]
\[\||f_{J_1}f_{J_2}|^\frac{1}{2}\|_{L^4(\mathbb{R}^2)}\lesssim\|(\sum_{I_1\in\mathcal{P}_1}|f_{I_1}|^2)^\frac{1}{4}(\sum_{I_2\in\mathcal{P}_2}|f_{I_2}|^2)^\frac{1}{4})\|_{L^4(\mathbb{R}^2)}.\]
[3, p39-p40]
\end{theorem}
\begin{proof}
Using the partitions, \(f_{J_j}=\sum_{I_j\in\mathcal{P}_j}f_{I_j}\), which hold as \(\mathcal{N}_{J_j}(\delta^2)=\bigcup\mathcal{N}_{I_j}(\delta^2)\), and \(\mathcal{N}_{I_j}(\delta^2)\) are disjoint. Further partition \(\mathcal{P}_j=\bigcup_{k=1}^{N+1}\mathcal{P}_{j,k}\), where the intervals of each \(\mathcal{P}_{j,k}\) are mutually separated by at least \(N\delta\).

By Plancherel's identity for inner products,
\begin{align*}
\||\sum_{I_1\in\mathcal{P}_{1,k}}f_{I_1}\sum_{I_2\in\mathcal{P}_{2,k'}}f_{I_2}|^\frac{1}{2}\|_{L^4(\mathbb{R}^2)}^4&=\int_{\mathbb{R}^2}|\sum_{I_1\in\mathcal{P}_{1,k}}f_{I_1}\sum_{I_2\in\mathcal{P}_{2,k'}}f_{I_2}|^2dx,&&\\
&=\sum_{I_1,I'_1\in\mathcal{P}_{1,k}}\sum_{I_2,I'_2\in\mathcal{P}_{2,k'}}\int_{\mathbb{R}^2}f_{I_1}\overline{f_{I_1'}}f_{I_2}\overline{f_{I_2'}}dx,&&\\
&=\sum_{I_1,I'_1\in\mathcal{P}_{1,k}}\sum_{I_2,I'_2\in\mathcal{P}_{2,k'}}\int_{\mathbb{R}^2}\widehat{f_{I_1}\overline{f_{I_1'}}}\,\overline{\widehat{\overline{f_{I_2}}f_{I_2'}}}dx.
\end{align*}
Note that \(\widehat{f_{I_1}\overline{f_{I_1'}}}=\widehat{f_{I_1}}*\widehat{\overline{f_{I_1'}}}\), \(\supp(\widehat{f_{I_1}})=\mathcal{N}_{I_1}(\delta^2)\) and \(\supp(\widehat{\overline{f_{I_1'}}})=-\mathcal{N}_{I_1'}(\delta^2)\), then \(\supp(\widehat{f_{I_1}\overline{f_{I_1'}}})\subseteq\mathcal{N}_{I_1}(\delta^2)-\mathcal{N}_{I_1'}(\delta^2)\). Similarly, \(\supp(\overline{\widehat{\overline{f_{I_2}}f_{I_2'}}})=\mathcal{N}_{I_2}(\delta^2)-\mathcal{N}_{I_2'}(\delta^2)\).

Since by assumption \(d(I_j,I_j')>N\delta\) for \(i=1,2\), by the previous theorem, the term \(\int_{\mathbb{R}^2}\widehat{f_{I_1}\overline{f_{I_1'}}}\,\overline{\widehat{\overline{f_{I_2}}f_{I_2'}}}dx\) vanishes unless either when \(d(I_1,I_2)\leqslant N\delta\) and \(d(I_1',I_2')\leqslant N\delta\), or \(I_j=I_j'\) for each \(j\).

For the first case, reformulate \(d(I_1,I_2')\leqslant N\delta\) and \(d(I_1',I_2)\leqslant N\delta\) as the collection of \((I_2,I_2')\) in a neighbourhood of \((I_1,I_1')\), denote \(\nei(I_1,I_1')\). Since \(I_2\in\mathcal{P}_{2,k'}\) are separated by at least \(N\delta\), each \(\nei(I_1,I_1')\) contains a fixed number of the pair \((I_2,I_2')\), independent of \(\delta\). Hence
\begin{align*}
&\sum_{I_1,I'_1\in\mathcal{P}_{1,k}}\sum_{I_2,I'_2\in\mathcal{P}_{2,k'}}\int_{\mathbb{R}^2}f_{I_1}\overline{f_{I_1'}}f_{I_2}\overline{f_{I_2'}}dx&&\\
&=\int_{\mathbb{R}^2}\sum_{I_1,I_1'\in\mathcal{P}_{1,k}}\sum_{I_2,I_2'\in \nei(I_1,I_1')}|f_{I_1}f_{I_1'}f_{I_2}f_{I_2'}|dx,&&\\
&\leqslant\int_{\mathbb{R}^2}(\sum_{I_1,I_1'\in\mathcal{P}_{1,k}}|\sum_{I_2,I_2'\in \nei(I_1,I_1')}f_{I_1}f_{I_1'}|^2)^\frac{1}{2}(\sum_{I_1,I_1'\in\mathcal{P}_{1,k}}|\sum_{I_2,I_2'\in \nei(I_1,I_1')}f_{I_2}f_{I_2'}|^2)^\frac{1}{2}dx,&&\\
&\leqslant\int_{\mathbb{R}^2}(\sum_{I_1,I_1'\in\mathcal{P}_{1,k}}|f_{I_1}f_{I_1'}|^2)^\frac{1}{2}(\sum_{I_1,I_1'\in\mathcal{P}_{1,k}}(\sum_{I_2,I_2'\in \nei(I_1,I_1')}|1f_{I_2}f_{I_2'}|)^2)^\frac{1}{2}dx,&&\\
&\lesssim\int_{\mathbb{R}^2}(\sum_{I_1,I_1'\in\mathcal{P}_{1,k}}|f_{I_1}f_{I_1'}|^2)^\frac{1}{2}(\sum_{I_1,I_1'\in\mathcal{P}_{1,k}}\sum_{I_2,I_2'\in \nei(I_1,I_1')}|f_{I_2}f_{I_2'}|^2)^\frac{1}{2}dx,&&\\
&\lesssim\int_{\mathbb{R}^2}(\sum_{I_1,I_1'\in\mathcal{P}_{1,k}}|f_{I_1}f_{I_1'}|^2)^\frac{1}{2}(\sum_{I_2,I_2'\in\mathcal{P}_{2,k}}|f_{I_2}f_{I_2'}|^2)^\frac{1}{2}dx,&&\\
&=\int_{\mathbb{R}^2}\sum_{I_1,I_1'\in\mathcal{P}_{1,k}}|f_{I_1}|^2\sum_{I_2,I_2'\in\mathcal{P}_{2,k}}|f_{I_2}|^2dx,&&\\
&\leqslant\|(\sum_{I_1\in\mathcal{P}_{1,k}}|f_{I_1}|^2)^\frac{1}{4}(\sum_{I_2\in\mathcal{P}_{2,k'}}|f_{I_2}|^2)^\frac{1}{4}\|_{L^4(\mathbb{R}^2)}^4.
\end{align*}
Here the second and third inequality follows by Cauchy-Schwartz inequality, and the fourth inequality follows by Cauchy-Schwartz inequality, with \(\sum_{I_2,I_2'\in \nei(I_1,I_1')}1\) equal to a constant as there are only finitely many pairs \(I_2,I_2'\) in \(\nei(I_1,I_1')\).
For the case \(I_j=I_j'\) for some \(j\),
\[\int_{\mathbb{R}^2}\sum_{I_1\in\mathcal{P_{1,k}}}\sum_{I_2\in\mathcal{P_{2,k'}}}f_{I_1}\overline{f_{I_1}}f_{I_2}\overline{f_{I_2}}dx=\|(\sum_{I_1\in\mathcal{P}_{1,k}}|f_{I_1}|^2)^\frac{1}{4}(\sum_{I_2\in\mathcal{P}_{2,k'}}|f_{I_2}|^2)^\frac{1}{4}\|_{L^4(\mathbb{R}^2)}^4.\]
The two cases together imply \[\||\sum_{I_1\in\mathcal{P}_{1,k}}f_{I_1}\sum_{I_2\in\mathcal{P}_{2,k'}}f_{I_2}|^\frac{1}{2}\|_{L^4(\mathbb{R}^2)}^4\leqslant\|(\sum_{I_1\in\mathcal{P}_{1,k}}|f_{I_1}|^2)^\frac{1}{4}(\sum_{I_2\in\mathcal{P}_{2,k'}}|f_{I_2}|^2)^\frac{1}{4}\|_{L^4(\mathbb{R}^2)}^4.\]
Finally since
\begin{align*}
\sum_{k,k'=1}^{N+1}\|(\sum_{I_1\in\mathcal{P}_{1,k}}|f_{I_1}|^2)^\frac{1}{4}(\sum_{I_2\in\mathcal{P}_{2,k'}}|f_{I_2}|^2)^\frac{1}{4}\|_{L^4(\mathbb{R}^2)}^4&=\sum_{k,k'=1}^{N+1}\int_{\mathbb{R}^2}(\sum_{I_1\in\mathcal{P}_{1,k}}|f_{I_1}|^2)(\sum_{I_2\in\mathcal{P}_{2,k'}}|f_{I_2}|^2)dx,&&\\
&=\int_{\mathbb{R}^2}\sum_{k,k'=1}^{N+1}(\sum_{I_1\in\mathcal{P}_{1,k}}|f_{I_1}|^2)(\sum_{I_2\in\mathcal{P}_{2,k'}}|f_{I_2}|^2)dx,&&\\
&=\int_{\mathbb{R}^2}(\sum_{I_1\in\mathcal{P}_1}|f_{I_1}|^2)(\sum_{I_2\in\mathcal{P}_2}|f_{I_2}|^2)dx,&&\\
&=\|(\sum_{I_1\in\mathcal{P}_1}|f_{I_1}|^2)^\frac{1}{4}(\sum_{I_2\in\mathcal{P}_2}|f_{I_2}|^2)^\frac{1}{4}\|_{L^4(\mathbb{R}^2)}^4,
\end{align*}
and
\begin{align*}
\sum_{k,k'=1}^{N+1}\||\sum_{I_1\in\mathcal{P}_{1,k}}f_{I_1}\sum_{I_2\in\mathcal{P}_{2,k'}}f_{I_2}|^\frac{1}{2}\|_{L^4(\mathbb{R}^2)}^2&=\sum_{k,k'=1}^{N+1}(\int_{\mathbb{R}^2}|\sum_{I_1\in\mathcal{P}_{1,k}}f_{I_1}\sum_{I_2\in\mathcal{P}_{2,k'}}f_{I_2}|^2)^\frac{1}{2},&&\\
&\gtrsim_N(\sum_{k,k'=1}^{N+1}\int_{\mathbb{R}^2}|\sum_{I_1\in\mathcal{P}_{1,k}}f_{I_1}\sum_{I_2\in\mathcal{P}_{2,k'}}f_{I_2}|^2)^\frac{1}{2},&&\\
&\gtrsim_N(\int_{\mathbb{R}^2}|\sum_{k,k'=1}^{N+1}\sum_{I_1\in\mathcal{P}_{1,k}}f_{I_1}\sum_{I_2\in\mathcal{P}_{2,k'}}f_{I_2}|^2)^\frac{1}{2},&&\\
&=\||\sum_{I_1\in\mathcal{P}_1}f_{I_1}\sum_{I_2\in\mathcal{P}_2}f_{I_2}|^\frac{1}{2}\|_{L^4(\mathbb{R}^2)}^2,
\end{align*}
we can put the three inequalities together,
\begin{align*}
\||\sum_{I_1\in\mathcal{P}_1}f_{I_1}\sum_{I_2\in\mathcal{P}_2}f_{I_2}|^\frac{1}{2}\|_{L^4(\mathbb{R}^2)}&\leqslant\sum_{k,k'=1}^{N+1}\||\sum_{I_1\in\mathcal{P}_{1,k}}f_{I_1}\sum_{I_2\in\mathcal{P}_{2,k'}}f_{I_2}|^\frac{1}{2}\|_{L^4(\mathbb{R}^2)}^2,&&\\
&\lesssim\sum_{k,k'=1}^{N+1}\|(\sum_{I_1\in\mathcal{P}_{1,k}}|f_{I_1}|^2)^\frac{1}{4}(\sum_{I_2\in\mathcal{P}_{2,k'}}|f_{I_2}|^2)^\frac{1}{4}\|_{L^4(\mathbb{R}^2)},&&\\
&=\|(\sum_{I_1\in\mathcal{P}_1}|f_{I_1}|^2)^\frac{1}{4}(\sum_{I_2\in\mathcal{P}_2}|f_{I_2}|^2)^\frac{1}{4}\|_{L^4(\mathbb{R}^2)}.
\end{align*}
\end{proof}
The second major component is the bilinear interaction of transverse wave packets [3, p32-p33]. A crucial concept is the class of rapidly decaying functions:
\begin{definition}[Rapid Decaying Functions]
Let A continuous function \(f:\mathbb{R}^n\to\mathbb{C}\) decays rapidly if for all \(M>0\),
\[f(x)\lesssim_{M,f}|x|^{-M}.\]
\end{definition}
Schwartz functions are also decaying rapidly, as \(\sup_{x\in\mathbb{R}^n}|x^\alpha(\partial^\beta f)(x)|<\infty\) for all multi-indices \(\alpha,\beta\in\mathbb{N}^n\) implies that let \(\beta=1\), \(f(x)\lesssim_{M,f}|x|^{-M}\) for all \(M=|\alpha|\).

 For a rectangular box \(B\) in \(\mathbb{R}^n\), there exists an affine map \(A_B:=L_B+v_B\) that maps \(B\) to \([-1,1]^n\), where \(L_B\) is an invertible matrix of scaling and rotation, and \(v_B\) is a vector of translation. Then for a function \(g\) supported on \(B\), it can be expressed as \(g(x)=f(A_B(x))\), where \(f\) is supported on \([-1,1]^n\). For the Fourier transform of \(g\),
 \begin{align*}
\widehat{g}(\xi)&=\int_{\mathbb{R}^n}f(A_B(x))e^{-2\pi i\xi\cdot x}dx,&&\\
&=\int_{\mathbb{R}^n}f(y)e^{-2\pi i\xi\cdot (y-v)}(\det L_B)^{-1} dy,&&\\
&=(\det L_B)^{-1} e^{2\pi i\xi\cdot L_B^{-1}v_B}\int_{\mathbb{R}^n}f(y)e^{-2\pi i\xi\cdot L_B^{-1}y} dy,&&\\
&=(\det L_B)^{-1} e^{2\pi i\xi\cdot L_B^{-1}v_B}\int_{\mathbb{R}^n}f(y)e^{-2\pi i(L_B^{-1})^\intercal\xi\cdot y} dy,&&\\
&=(\det L_B)^{-1} e^{2\pi i\xi\cdot L_B^{-1}v_B}\widehat{f}((L_B^{-1})^\intercal\xi).
 \end{align*}
Note that since \(L_B\) is invertible, then \((L_B^{-1})^\intercal=(L_B^\intercal)^{-1}\). Also, let \(e_j,e_k\) be unit vectors parallel to the sides of \(L_B^{-1}[-1,1]^n\), then by orthogonality of \(e_j,e_k\),
\[(L_B^{-1}e_j)\cdot(L_B^\intercal e_k)=e_j^\intercal L_B^{-\intercal}L_B^\intercal e_k=e_j\cdot e_k=\delta_{j,k}.\]
\begin{definition}[Dual Rectangles]
The n-dimensional rectangles \(B_1\), \(B_2\) with side length \((l_1^{(1)},\cdots,l_n^{(1)})\),  \((l_1^{(2)},\cdots,l_n^{(2)})\) mutually parallel, are dual rectangles if \( (l_j^{(1)}l_j^{(2)})=1\) for all \(1\leqslant j\leqslant n\).
\end{definition}
So clearly, \(L_B^{-1}[-1,1]^n\), which has the shape and orientation as \(B\), is a dual box of \(L_B^\intercal[-1,1]^n\).

A function \(g(x)=f(A_B(x))\) rapidly decays outside of a box \(B\) if \(f\) decays rapidly, that is, \(g(x)\lesssim_{M,g}|A_Bx|^{-M}\). Let \(g=f(A_B(x))\) be a Schwartz function supported in \(2B\) with rapid decay outside of \(B \), where \(B=A_B^{-1}[-1,1]^n\), then \(\widehat{g}(\xi)=(\det L_B)^{-1}e^{2\pi i\xi\cdot L_B^{-1}v_B}\widehat{f}(L_B^\intercal\xi)\). Since \(f\) is Schwartz, \(\widehat{f}\) is also Schwartz hence decays rapidly. So \(\widehat{f}((L_B^{-1})^\intercal\xi)\) rapidly decays outside of \(L_B^\intercal[-1,1]^n\), which is a dual box of \(L_B^{-1}[-1,1]^n\) and \(B\). So 
\[\widehat{g}(\xi)\leqslant|\det L_B|^{-1}\widehat{f}(L_B^\intercal\xi)\lesssim_M|\det L_B|^{-1}|L_B^\intercal x|^{-M},\] decays rapidly outside of \(L_B^\intercal[-1,1]^n\).

The previous discussions provide enough ingredients to construct the wave packets needed for the next theorem.

Let \(B\) be a rectangular box in \(\mathbb{R}^n\). There exists a Schwartz function \(V_T\) such that \(\supp(V_T)\subset-2B\). By the previous discussion, \(V_T\) has the form \(V_T=f(A(x))\), where \(A=L+v\) is the affine map that maps \(-2B\) to \([-1,1]^n\). Note that \(\widehat{V_T}\) is rapidly decaying outside \(L^\intercal[-1,1]^n\), which is has the shape and orientation of a dual box of \(2B\), centred on the origin. 

Hence let \(\mathcal{T}_B\) be a tessellation of \(\mathbb{R}^n\) with rectangular boxes \(T\in\mathcal{T}_B\) dual to \(2B\), then for each \(T\in\mathcal{T}_B\)  there exists \(u_T\in\mathbb{R}^n\) such that \(u_T\) is the centre of \(T\). Let \(k=(\frac{2^n}{|B|})^\frac{1}{2}\), and further let \(V_T':=e^{2\pi iu_T}V_T\), then 
\[\widehat{V_T'}=(\det L)^{-1}e^{2\pi i\xi\cdot L^{-1}v}\widehat{f}(L^{-\intercal}(\xi-u_T)).\]
Then the function \(W_T:=k\widehat{V_T'}\) satisfies \(\supp(\widehat{W_T})=\supp(kV_T'(-\cdot))\subset2B\) and \(|W_T|\lesssim_M|T|^{-\frac{1}{2}}(\chi_T)^M\) for all \(M\geqslant1\), where \(\chi(x):=(1+|x|)^{-100n}\), \(\chi_B:=\chi\circ A_B\).

Taking \(T\in\mathcal{T}_B\), the collection of \(W_T\) satisfies the previous criterion. For each \(w_T\in\mathbb{C}\), by Plancherel's identity, and that \(V_T'\) is a fixed Schwartz function by construction,
 \[\|\sum_{T\in\mathcal{T}_B}w_TW_T\|_{L^2}=\|\sum_{T\in\mathcal{T}_B}w_T\widehat{W_T}\|_{L^2}=\int_{\mathbb{R}^n}|\sum_{T\in\mathcal{T}_B}w_TkV_T'(-x)|^2\sim\sum_{T\in\mathcal{T}_B}|w_T|^2=\|w_t\|_{l_2}^2.\]

By the previous discussion, let \(F\) be a function such that \(\supp(\widehat{F})\subset B\), Taking \(W_T\) such that \(\widehat{W}_T=|T|^{\frac{1}{2}}e^{2\pi iu_Tx}\eta\), where \(\eta\in S(R)\) satisfies \(\mathbbm{1}_B\leqslant\eta\leqslant\mathbbm{1}_{2B}\). Since \(\supp(\widehat{W_T})\subset2B\) and \(\supp(\widehat{F})\subset B\), then \(\widehat{W_T}\) can be used to construct the Fourier expansion of \(\widehat{F}\) , and the series is supported on \(B\), in the sense that
\[\widehat{F}=|T|\sum_{T\in\mathcal{T}_B}\langle\widehat{F},e^{2\pi iu_Tx}\eta\rangle e^{2\pi iu_T}\eta=\sum_{T\in\mathcal{T}_B}\langle\widehat{F},\widehat{W_T}\rangle \widehat{W_T},\]since the Fourier transform is self-adjoint, then \(\langle\widehat{F},\widehat{W_T}\rangle =\langle F,W_T\rangle\), and hence for all \(F\) such that \(\supp(\widehat{F})\subset B\), 
\
\[\widehat{F}=\sum_{T\in\mathcal{T}_B}\langle F,W_T\rangle \widehat{W_T},\,\,\,\,\,\,\,\,\,\,\,\,\,\,\,\,F=\sum_{T\in\mathcal{T}_B}\langle F,W_T\rangle W_T.\]
The following lemma helps the proof in the sense that the disjointness of \(J_1\) and \(J_2\) in Restriction Theorem 4 ensures the rectangles containing their caps forms an nonzero angle between their short sides.
\begin{lemma}[Bilinear Interaction of Transverse Wave packets]
Let \(B_1\),\(B_2\) be two rectangular boxes in \(\mathbb{R}^n\), \(n\geqslant2\), with the short side of length \(\delta^2\), and the other \(n-1\) sides having length \(\delta\leqslant1\). Assume that the angle \(\nu\) between the axes corresponding to their short sides that satisfies \(\nu\gg\delta\). Assume also that \(f_1,f_2:\mathbb{R}^n\to\mathbb{C}\) have Fourier transforms supported inside \(B_1,B_2\) respectively, then
\(\||f_1f_2|^\frac{1}{2}\|_{L^4(\mathbb{R}^n)}\lesssim_\nu\delta^\frac{n+2}{4}(\|f_1\|_{L^2(\mathbb{R}^n)}\|f_2\|_{L^2(\mathbb{R}^n)})^\frac{1}{2}\).[3, p41-p42]
\end{lemma}
\begin{proof}
By the previous discussions, for each \(f_j\), there exists a wave packet decomposition
\[f_j=\sum_{T\in\mathcal{T}_i}w_TW_T,\]
where \(T\) are rectangles of dimension \(\sim\delta^{-1}\times\cdots\times\delta^{-1}\times\delta^{-2}\) in a tesselation \(\mathcal{T}_j\) of \(\mathbb{R}^n\), that are dual rectangles of \(B_j\). The wave packets are called transverse as the axes corresponding to their short sides are not parallel. The wave packet decompositions further satisfy \(\supp(\widehat{W_T})\subset 2B_j\), \(\|W_T\|_{L^2}\sim1\), and that 
\[|W_T|\lesssim|T|^{-\frac{1}{2}}\chi_T(x),\,\,\,\,\,\,\,\,\sum_{T\in\mathcal{T}_j}|w_T|^2\sim\|f_j\|_2^2.\]
Hence by the triangle inequality,
\begin{align*}
\||f_1f_2|^\frac{1}{2}\|_{L^4(\mathbb{R}^n)}^4&=\int_{\mathbb{R}^n}|\sum_{T_1\in\mathcal{T}_1}\sum_{T_2\in\mathcal{T}_2}w_{T_1}w_{T_2}W_{T_1}W_{T_2}|^2dx,&&\\
&=\int_{\mathbb{R}^n}\sum_{T_1,T_1'\in\mathcal{T}_1,\,T_2,T_2'\in\mathcal{T}_2}|w_{T_1}w_{T_1'}w_{T_2}w_{T_2'}W_{T_1}W_{T_1'}W_{T_2}W_{T_2'}|dx,&&\\
&\leqslant\sum_{T_1,T_1'\in\mathcal{T}_1,\,T_2,T_2'\in\mathcal{T}_2}|w_{T_1}w_{T_1'}w_{T_2}w_{T_2'}|\int_{\mathbb{R}^n}|W_{T_1}W_{T_1'}W_{T_2}W_{T_2'}|dx.
\end{align*}
By the assumption on angles between \(B_1,B_2\), the angle between the axes corresponding to the long sides of \(T_1\in\mathcal{T}_1\) and \(T_2\in\mathcal{T}_2\) are also \(\nu\), where \(\nu\gg\delta\). Denote the unit vectors of the long sides of  \(T_1\in\mathcal{T}_1\) and \(T_2\in\mathcal{T}_2\) as \(a,b\), the angle \(\nu\) is determined by \(\cos(\nu)=a\cdot b\).

Note that the intersections \(\mathcal{I}= T_1\cap T_2\), \(T_1\in\mathcal{T}_1\), \(T_2\in\mathcal{T}_2\) have volume \(\vol(\mathcal{I})\lesssim_\nu\delta^{-n}\). This can be seen inductively. For the base case \(n=2\), since for all possible spatial placements of \(T_1,T_2\), \(\mathcal{I}\) is contained in a rhombus of height \(\delta^{-1}\) and two of the angles being \(\nu\), which implies that 
\[\vol(\mathcal{I})=(\csc(\nu)\delta^{-1})\delta^{-1}\lesssim_\nu\delta^{-2}.\]
For rectangles of higher dimensions, all additional dimensions of \(T_j\) are \(\sim\delta^{-1}\). For the \(3\)-dimensional case, \(\mathcal{I}\) is the parallelepiped which has the base being the rhombus contained in the \(2\)-dimensional case, and height being \(\delta^{-1}\), which implies that 
\[\vol(\mathcal{I})=(\csc(\nu)\delta^{-1})\delta^{-1}\delta^{-1}\lesssim_\nu\delta^{-3}.\]
By induction, for the \(n\)-dimensional case, \(\vol(\mathcal{I})\lesssim_\nu\delta^{-n}\). 

Since \(W_{T_1},W_{T_2}\) decay rapidly outside of \(T_1,T_2\) respectively, \(W_{T_1}W_{T_1'}\) and \(W_{T_2}W_{T_2'}\) are Schwartz functions that decay rapidly outside of \(T_1\cap T_1'\) and \(T_2\cap T_2'\) respectively. Hence for each \(T_1,T_1'\in\mathcal{T}_1\) and \(T_2,T_2'\in\mathcal{T}_2\),
\[\int_{\mathbb{R}^n}|W_{T_1}W_{T_1'}W_{T_2}W_{T_2'}|\lesssim_\nu\delta^{n+2}c(T_1,T_1')c(T_2,T_2'),\]
where \(c(T_j,T_j')\geqslant0\) are weight constants that decrease rapidly as the distance between the centres of \(T_j,T_j'\) increases.

Together with Cauchy-Schwartz inequality,
\begin{align*}
\||f_1f_2|^\frac{1}{2}\|_{L^4(\mathbb{R}^n)}^4&\lesssim_\nu\delta^{n+2}\sum_{T_1,T_1'\in\mathcal{T}_1}|w_{T_1}w_{T_1'}c(T_1,T_1')|\sum_{T_2,T_2'\in\mathcal{T}_1}|w_{T_2}w_{T_2'}c(T_2,T_2')|,&&\\
&\leqslant\delta^{n+2}\sum_{T_1,T_1'\in\mathcal{T}_1}|w_{T_1}w_{T_1'}|\sum_{T_2,T_2'\in\mathcal{T}_1}|w_{T_2}w_{T_2'}|,&&\\
&\leqslant4\delta^{n+2}\sum_{T_1\in\mathcal{T}_1}|w_{T_1}|^2\sum_{T_2\in\mathcal{T}_1}|w_{T_2}|^2,&&\\
&\sim\delta^{n+2}\|f_1\|_{L^2}^2\|f_2\|_{L^2}^2.\,
\qedhere
\end{align*}
\end{proof}
For the 2-dimensional case of the bilinear \(L^2\) restriction problem, the cubes \(\Omega_1,\Omega_2\) become intervals in \([-1,1]\), denoted as \(J_1,J_2\). Then the restriction problem becomes
\[\int_{\mathbb{R}^2}|\mathcal{E}_{J_1}f\mathcal{E}_{J_2}f|^2=\||\mathcal{E}_{J_1}f\mathcal{E}_{J_2}f|^\frac{1}{2}\|_{L^4(\mathbb{R}^n)}^4\lesssim\|f\|_{L^2(J_1)}^2\|f\|_{L^2(J_2)}^2.\]
Let \(R\gg1\) be an arbitrary large number such that \(R\geqslant(d(J_1,J_2))^\frac{1}{2}\), if \[\int_{[-R,R]^2}|\mathcal{E}_{J_1}f\mathcal{E}_{J_2}f|^2\lesssim\|f\|_{L^2(J_1)}^2\|f\|_{L^2(J_2)}^2\]
holds with the inequality constant independent of \(R\), then sending \(R\) to infinity will prove the bilinear \(L^2\) restriction estimate.

Let \(\eta_R:\mathbb{R}^2\to[0,\infty]\) be a function with support \(\supp(\widehat{\eta_R})=B(0,\frac{1}{R})\), such that \(\mathbbm{1}_{[-R,R]^2}\leqslant\eta_R\) and \(\int_\mathbb{R}\sup_{x_1\in\mathbb{R}}\eta_R(x_1,x_2)^2dx_2\lesssim R\).

To construct such a function \(\eta_R\), let \(\eta_1\in\mathcal{S}(\mathbb{R}^2)\) non negative with \(\supp(\eta_1)\subset B(0,\frac{1}{2R})\), then \(\eta_2:=\eta_1*\eta_1\) has support \(\supp(\eta_2)\subset\overline{B(0,\frac{1}{R})}\). Hence
\[\widehat{\eta_2}(0)=(\int_{\mathbb{R}^2}\eta_1e^{-2\pi i\cdot0})^2=(\int_{\mathbb{R}^2}\eta_1)^2>0.\]
Since both \(\eta_2\) and \(\widehat{\eta_2}\) are Schwartz functions, they are continuous and there exists a neighbourhood of \(0\) such that \(\widehat{\eta_2}>\frac{1}{2}\widehat{\eta_2}(0)\) holds. Hence there exists large constants \(C_1>1, C_2>0\) such that \(\min_{x\in[-R,R]^2}C_2\widehat{\eta_2}(\frac{x}{C_1})\geqslant1\). Then the Schwartz function \(\eta_R(y):=C_2\widehat{\eta_2}(\frac{x}{C_1})\) has Fourier transform \(\widehat{\eta_2}(y)=C_2C_1^2\eta_R(C_1y)\) supported in \(\overline{B(0,\frac{1}{C_1R})}\subset\overline{B(0,\frac{1}{R})}\).

The condition \(\int_\mathbb{R}\sup_{x_1\in\mathbb{R}}\eta_R(x_1,x_2)^2dx_2\lesssim R\) follows from the boundness property of Schwartz functions, and that \(\sup_{x_1\in\mathbb{R}}\eta_R(x_1,x_2)^2\) decays rapidly outside of \([-R,R]\).

Hence 
\begin{align*}
\int_{[-R,R]^2}|\mathcal{E}_{J_1}f\mathcal{E}_{J_2}f|^2&=\int_{\mathbb{R}^2}|\mathcal{E}_{J_1}f\mathcal{E}_{J_2}f|^2\mathbbm{1}_{[-R,R]^2},&&\\
&\leqslant\int_{\mathbb{R}^2}|\eta_R\cdot(\mathcal{E}_{J_1}f)\eta_R\cdot(\mathcal{E}_{J_2}f)|^2.
\end{align*}
Since \(\supp(\eta_R(\mathcal{E}_{J_j}f))\subset\mathcal{N_{J_j'}}(R^{-1})\), where \(J_j'\) is slightly larger than \(J_j\). Partition \(J_j\) into intervals \(I_j\) of length \(R^{-\frac{1}{2}}\), then \(\eta_R\mathcal{E}_{J_j}f=\sum_{I_j\in\mathcal{P}_j}\eta_R\mathcal{E}_{I_j}f\). By the reverse square function inequality, 
\[\int_{\mathbb{R}^2}|\eta_R(\mathcal{E}_{J_1}f)\eta_R(\mathcal{E}_{J_2}f)|^2\lesssim\int_{\mathbb{R}^2}\sum_{I_1\in\mathcal{P}_1}|\eta_R(\mathcal{E}_{J_1}f)|^2\sum_{I_2\in\mathcal{P}_2}|\eta_R(\mathcal{E}_{J_2}f)|^2.\]
Here \(\supp(\eta_R\mathcal{E}_{I_1}f)\) is contained in a rectangle \(B_j\) with dimension similar to \(R^{-\frac{1}{2}}\times R^{-1}\). Since \(J_1\), \(J_2\) are disjoint, \(\widehat{\eta_R\mathcal{E}_{I_1}f}\) and \(\widehat{\eta_R\mathcal{E}_{I_2}f}\) are also disjoint, hence the short sides of \(B_1\), \(B_2\) forms an angle of \(\nu>0\) depending on \(d(J_1,J_2)\), hence by Lemma 3.3, let \(\delta\sim R^{-\frac{1}{2}}\), then
\[\int_{\mathbb{R}^2}|\eta_R(\mathcal{E}_{I_1}f)\eta_R(\mathcal{E}_{I_2}f)|^2\lesssim R^{-2}\|\eta_R\mathcal{E}_{I_1}f\|_{L^2}^2\|\eta_R\mathcal{E}_{I_2}f\|_{L^2}^2.\]
Using Fubini's theorem, 
\begin{align*}
\|\eta_R\mathcal{E}_{I_j}f\|_{L^2}^2&=\iint_\mathbb{R}|\eta_R\mathcal{E}_{I_j}f|^2dx_1dx_2,&&\\
&\leqslant\iint_\mathbb{R}|\sup_{x_1\in\mathbb{R}}|\eta_R(x_1,x_2)|\mathcal{E}_{I_j}f(x_1,x_2)|^2dx_1dx_2,&&\\
&\leqslant(\sup_{x_2\in\mathbb{R}}\int_\mathbb{R}|\mathcal{E}_{I_j}f(x_1,x_2)|^2dx_1)\int_\mathbb{R}\sup_{x_1\in\mathbb{R}}|\eta_R(x_1,x_2)|^2dx_2.
\end{align*}
Note that \(\int_{\mathbb{R}}|\mathcal{E}^\psi(f(x_1,x_2))|^2dx_1=\|f\|_{L^2}^2\) for all \(x_2\). By definition

The construction of \(\eta_R\) implies \(\|\eta_R\mathcal{E}_{I_j}f\|_{L^2}^2\leqslant R\|f\|_{L^2(I_j)}^2\).

Using the four inequalities derived above, 
\begin{align*}
\int_{[-R,R]^2}|\mathcal{E}_{J_1}f\mathcal{E}_{J_2}f|^2&\leqslant\int_{\mathbb{R}^2}|\eta_R(\mathcal{E}_{J_1}f)\eta_R(\mathcal{E}_{J_2}f)|^2,&&\\
&\lesssim\int_{\mathbb{R}^2}\sum_{I_1\in\mathcal{P}_1}|\eta_R(\mathcal{E}_{J_1}f)|^2\sum_{I_2\in\mathcal{P}_2}|\eta_R(\mathcal{E}_{J_2}f)|^2,&&\\
&\lesssim R^{-2}\|\eta_R\mathcal{E}_{I_1}f\|_{L^2}^2\|\eta_R\mathcal{E}_{I_2}f\|_{L^2}^2,&&\\
&\lesssim\|f\|_{L^2(\Omega_1)}^2\|f\|_{L^2(\Omega_2)}^2.
\end{align*}
for arbitrary \(R\), which proves the Restriction Theorem 4.[3, p43-p44]

\subsection{Application of the Bilinear Restriction Theorem}
The bilinear restriction theorem is useful in extending the restriction range derived from the Stein-Tomas Theorem. [3, p79-p84]

\begin{definition}[Bilinear Restriction Constants]
Let \(1\leqslant p,q\leqslant\infty\), \(0<D\leqslant1\), define the bilinear restriction constant \(\BR_{\mathbb{P}^{n-1}}^*(q\times q\mapsto p, D)\) as the minimal constant \(C\) such that for all cube pairs \(\Omega_1,\Omega_2\subset[-1,1]^{n-1}\) with \(d(\Omega_1,\Omega_2)\geqslant D\), \(f:\Omega_1\cup\Omega_2\to\mathbb{C}\),
\[\||\mathcal{E}_{\Omega_1}f\mathcal{E}_{\Omega_2}f|^\frac{1}{2}\|_{L^p(\mathbb{R}^n)}\leqslant C(\|f\|_{L^q(\Omega_1)}\|f\|_{L^q(\Omega_2)})^\frac{1}{2}.\]
[3, p79]
\end{definition}

\begin{theorem}
Let \(\Omega_1,\Omega_2\) be two cubes in \([-1,1]^{n-1}\) with sides parallel to the axis, and side length \(\delta\) . If \(D:=d(\Omega_1,\Omega_2)\geqslant4\delta\), then for all \(1\leqslant p,q\leqslant\infty\) and \(f:\Omega_1\cup\Omega_2\to\mathbb{C}\), then
\[\||\mathcal{E}_{\Omega_1}f\mathcal{E}_{\Omega_2}f|^\frac{1}{2}\|_{L^p(\mathbb{R}^n)}\leqslant D^{\frac{n-1}{q'}-\frac{n+1}{p}}\BR_{\mathbb{P}^{n-1}}^*(q\times q\mapsto p, \frac{1}{2})(\|f\|_{L^q(\Omega_1)}\|f\|_{L^q(\Omega_2)})^\frac{1}{2}.\]
[3, p80]
\end{theorem}
\begin{proof}
Let \(\xi_0\) be the midpoint of the line segment connecting the centres of \(\Omega_1\), \(\Omega_2\). Define the affine map on \(\mathbb{R}^{n-1}\) as \(L(\xi)=L_{\xi_0,D(\xi)}:=\frac{\xi-\xi_0}{D}\), then using change of variables with \(\eta:=\frac{\xi-\xi_0}{D}\), and let \(f_L:=f\circ L^{-1}\),
\begin{align*}
|\mathcal{E}_{\Omega_j}f(\overline{x},x_n)|&=|\int_{\Omega_j}f(\xi)e^{2\pi i(\overline{x}\cdot\xi+x_n|\xi|^2)}d\xi|,&&\\
&=|e^{2\pi i(\overline{x}\cdot\xi_0+x_n|\xi_0|^2)}\int_{\Omega_j}f(\xi)e^{2\pi i(\overline{x}\cdot\xi+x_n|\xi|^2-\overline{x}\cdot\xi_0-x_n|\xi_0|^2)}d\xi|,&&\\
&=|e^{2\pi i(\overline{x}\cdot\xi_0+x_n|\xi_0|^2)}\int_{\Omega_j}f(\xi)e^{2\pi i((\overline{x}+2x_n\cdot\xi_0)\cdot(\xi-\xi_0)+x_n|\xi-\xi_0|^2))}d\xi|,&&\\
&=|D^{n-1}e^{2\pi i(\overline{x}\cdot\xi_0+x_n|\xi_0|^2)}\int_{\Omega_j}f_L(\eta)e^{2\pi i(D(\overline{x}+2x_n\cdot\xi_0)\cdot\eta+D^2x_n|\eta|^2))}d\eta|,&&\\
&=D^{n-1}|\mathcal{E}_{L(\Omega_j)}f_L(D(\overline{x}+2x_n\xi_0),D^2x_n)|.
\end{align*}

Since \(L\) maps every point \(x\in\Omega_1,\Omega_2\) by a shift \(y=x-\xi_0\) and a dilation \(z=\frac{y}{D}\), \(d(\Omega_1,\Omega_2)\) does not change after the shift, and is scaled to 1 after the dilation. Also, \(L(\Omega_1),L(\Omega_2)\) are symmetric around the origin and, under the assumption that \(D\geqslant4\delta\), \(L(\Omega_1),L(\Omega_2)\) are in \([-1,1]^{n-1}\). Using change of variables,
\begin{align*}
&\||\mathcal{E}_{\Omega_1}f\mathcal{E}_{\Omega_2}f|^\frac{1}{2}\|_{L^p(\mathbb{R}^n)}&&\\
&=D^{n-1-\frac{n+1}{p}}\||\mathcal{E}_{L(\Omega_1)}f\mathcal{E}_{L(\Omega_2)}f|^\frac{1}{2}\|_{L^p(\mathbb{R}^n)},&&\\
&\leqslant D^{n-1-\frac{n+1}{p}}\BR_{\mathbb{P}^{n-1}}^*(q\times q\mapsto p, \frac{1}{2})(\|f\|_{L^q(L(\Omega_1))}\|f\|_{L^q(L(\Omega_2))})^\frac{1}{2},&&\\
&\leqslant D^{\frac{n-1}{q'}-\frac{n+1}{p}}\BR_{\mathbb{P}^{n-1}}^*(q\times q\mapsto p, \frac{1}{2})(\|f\|_{L^q(\Omega_1)}\|f\|_{L^q(\Omega_2)})^\frac{1}{2}.
\qedhere
\end{align*}
\end{proof}
When \(p,q\) are in the restriction range \(\frac{n-1}{q'}\geqslant\frac{n+1}{p}\), then \(\frac{n-1}{q'}-\frac{n+1}{p}\) is non-negative and \( D^{\frac{n-1}{q'}-\frac{n+1}{p}}\) decreases as \(D\) decreases.

Note that Theorem 3.4 will be useful when combined with dyadic cubes as their sides are also parallel with the axis. A dyadic cube has the form \([l_12^k,(l_1+1)2^k]\times[l_22^k,(l_2+1)2^k]\times\cdots\) , where \(l_j\) are integers. The Whitney Decomposition of open sets proposes a kind of covering of open sets using dyadic cubes, and the version used in this dissertation has a consequence which also requires cubes to be separated as in Theorem 3.4.
\begin{theorem}
Let \(S\) be a closed set in \(\mathbb{R}^m\). Then there exists a collection \(\mathcal{C}\) of dyadic cubes \(\Omega\) with their interiors pairwise disjoint, such that \(\mathbb{R}^m\setminus S=\bigcup_{\Omega\in\mathcal{C}}\Omega\) and \(4l(\Omega)\leqslant d(\Omega,S)\leqslant50l(\Omega)\), where \(l(\Omega)\) is the side length of \(\Omega\).[3, p81]
\end{theorem}
\begin{proof}
Let \(\mathcal{C}' \) be the collection of all dyadic cubes \(\Omega\in\mathbb{R}^m\) such that \(8\Omega\cap S=\emptyset\) where \(8\Omega\) is the dilation of \(\Omega\) from its centre by a factor of \(8\). 

Let \(\mathcal{C}\) consist of the cubes in \(\mathcal{C}' \) that are maximal with respect to inclusion, that is, for any dyadic cube \(\Omega\in\mathcal{C}\), if there exists a dyadic cube \(\Omega'\in\mathcal{C}'\) such that \(\Omega\subset\Omega'\), then \(\Omega=\Omega'\). 

Then since \(S\) is closed, hence \(\mathbb{R}^m\setminus S\) is open, so for any point \(x\in\mathbb{R}^m\setminus S\), there exists an open ball \(B(x,r)\) such that \(B(x,r)\cap S=\emptyset\). Let \(8r'\leqslant r\) where \(r\) is an integer power of \(2\), then \(x\) lies in an dyadic cube \(\Omega'\) of length \(r'\), and \(8\Omega'\subset B(x,r)\) implying that \(\Omega'\in \mathcal{C}'\) and hence there exists an \(\Omega\in\mathcal{C}\) that contains \(x\). So \(\mathcal{C}\) satisfies \(\mathbb{R}^m\setminus S=\bigcup_{\Omega\in\mathcal{C}}\Omega\).

Any two different dyadic cubes either have disjoint interiors or one cube contains the other, by construction, any two different dyadic cubes in \(\mathcal{C}\) have disjoint interiors. Denote \(l(\Omega)\) the side length of the dyadic cube \(\Omega\in \mathcal{C}\), then the condition \(4l(\Omega)\leqslant d(\Omega,S)\) holds by the construction of \(\mathcal{C}\), and if \(d(\Omega,S)\geqslant50l(\Omega)\), then there exists an \(\Omega'\in\mathcal{C'}\) that satisfies \(\Omega'\neq\Omega\) and \(\Omega\subset\Omega'\).
\end{proof}
The previously mentioned consequence of the Whitney decomposition is that:
\begin{theorem}
For \(n\geqslant2\), there is a collection \(\mathcal{C}\) of closed cubes \(\Omega=\Omega_1\times\Omega_2\subset[-1,1]^{n-1}\times[-1,1]^{n-1}\) with their interiors pairwise disjoint, such that
\[[-1,1]^{2n-2}\setminus\{(\xi,\xi):\xi\in[-1,1]^{n-1}\}=\bigcup_{\Omega\in\mathcal{C}}\Omega,\] and \(4l(\Omega)\leqslant d(\Omega_1,\Omega_2)\leqslant100l(\Omega)\).[3, p81-p82]
\end{theorem}
\begin{proof}
\([0,2]^{2n-2}\) is dyadic. Let \(\mathcal{C}'\) be as in the previous proof with \(m=2n-2\) and \(S:=\{(\xi,\xi):\xi\in\mathbb{R}^{n-1}\}\), and further take the sub-collection of cubes in \([0,2]^{2n-2}\). Let \(\mathcal{C}\) consist of the cubes in \(\mathcal{C}' \) that are maximal with respect to inclusion. Then \(\mathcal{C}\) covers \([0,2]^{2n-2}\), and for any two cubes \(\Omega_1,\Omega_2\), the inequality \(4l(\Omega)\leqslant d(\Omega_1,\Omega_2)\leqslant100l(\Omega)\) holds by the previous theorem, as \(d(\Omega_1,\Omega_2)\leqslant d(\Omega_1,S)+d(\Omega_2,S)\leqslant100l(\Omega)\). Each pair \(\Omega_1,\Omega_2\) of \(\Omega\) are disjoint by construction,  and \(4l(\Omega)\leqslant d(\Omega_1,\Omega_2)\) by construction.

Hence applying a translation of \(\Omega_1-[-1,\cdots,-1]\), the theorem holds.
\end{proof}

\begin{lemma}
Let \(\mathcal{R}\) be a finite collection of rectangular boxes in \(\mathbb{R}^n\), with each boxes in \(\mathcal{R}\) are still pairwise disjoint after dilating around the centre by a factor of \(2\). Let \(F_R:\mathbb{R}^n\to\mathbb{C}\) with \(\supp(\widehat{F_R})\subset R\) and \(F_R\in L^s\), for some \(1\leqslant s\leqslant\infty\), then

\(\|\sum_RF_R\|_{L^s}\lesssim(\sum_R\|F_R\|_{L^s}^s)^\frac{1}{s}\) for \(1\leqslant s\leqslant2\), and
\(\|\sum_RF_R\|_{L^s}\lesssim(\sum_R\|F_R\|_{L^s}^{s'})^\frac{1}{s'}\) for \(s\geqslant2\).[3, p82]
\end{lemma}
\begin{proof}
Let \(\varphi_R\in\mathcal{S}(\mathbb{R}^n)\) with \(\mathbbm{1}_R\leqslant\widehat{\varphi_R}\leqslant\mathbbm{1}_{2R}\), then \(\|\varphi_R\|_{L^1}\) is also bounded. Since \(\widehat{F_R}=\widehat{F_R*\varphi_R}=\widehat{F_R}\widehat{\varphi_R}\),  \(\supp(\widehat{F_R})\subset R\) implies that \(F_R=F_R*\varphi_R\).

Let \(T\) be an operator acting on a family of functions \(G_\mathcal{R}=\{G_R\}_{R\in\mathcal{R}}\) with
\[T(G_\mathcal{R})=\sum_RG_R*\varphi_R.\]
By Young's convolution inequality,
\[\|T(F_\mathcal{R})\|_{L^1}=\|\sum_RF_R*\varphi_R\|_{L^1}\leqslant\sum_R\|F_R*\varphi_R\|_{L^1}\leqslant\sum_R\|F_R\|_{L^1}\|\varphi_R\|_{L^1}\sim\sum_R\|F_R\|_{L^1}.\]
Similarly, 
\[\|T(F_\mathcal{R})\|_{L^\infty}\lesssim\sum_R\|F_R\|_{L^\infty}.\]
Since the collection of \(\widehat{F_R}\) are pairwise orthogonal, then by Plancherel's identity,
\[\|T(F_\mathcal{R})\|_{L^2}^2=\|\sum_RF_R\|_{L^2}^2=\|\widehat{\sum_RF_R}\|_{L^2}^2=\|\sum_R\widehat{F_R}\|_{L^2}^2=\sum_R\|\widehat{F_R}\|_{L^2}^2=\sum_R\|F_R\|_{L^2}^2.\]
Since \(\mathcal{R}\) is a finite collection, \(T\) is a bounded map on \(L^r(\mathbb{R}^n,l^r(\mathcal{R}))\mapsto L^r(\mathbb{R}^n)\) for \(r=1,2\), and \(T\) is also a bounded map on  \(L^\infty(\mathbb{R}^n,l^1(\mathcal{R}))\mapsto L^\infty(\mathbb{R}^n)\), where \(L^s(\mathbb{R}^n,l^s(\mathcal{R}))\) is the space of strongly measurable functions \(F_{\mathcal R} : \mathbb{R}^n \to l^s(\mathcal{R})\) such that \( \sum_R \| F_R \|^s_{L^s} < \infty\) for \(1\leqslant s<\infty\), and \(L^\infty(\mathbb{R}^n,l^1(\mathcal{R}))\) is the space of strongly measurable functions such that \(\sum_R |F_R(x)|\) is bounded almost everywhere. [13, p21]

The space \(L^s(\mathbb{R}^n,l^s(\mathcal{R}))\) is a Bochner space with norm
\[\|F_R\|_{L^s(\mathbb{R}^n,l^s(\mathcal{R}))}:=((\int \|   \{ F_R(x) \} \|^s_{l^s(\mathcal R)} dx)^\frac{1}{s}=\int \sum_R| F_R(x) |^s dx=\sum_R \| F_R \|^s_{L^s}\]
for \(1\leqslant s<\infty\), and
\[\|F_R\|_{L^\infty(\mathbb{R}^n,l^1(\mathcal{R}))}:=\inf\{r\geqslant0,\mu\{\|F_R\|_{L^\infty}>r\}=0\}.\]
Using the Riesz-Thorin interpolation theorem on Bochner spaces[13, p83-p84], \(\|\sum_RF_R\|_{L^s}\lesssim(\sum_R\|F_R\|_{L^s}^s)^\frac{1}{s}\) for \(1\leqslant s\leqslant2\), and \(\|\sum_RF_R\|_{L^s}\lesssim(\sum_R\|F_R\|_{L^s}^{s'})^\frac{1}{s'}\) for \(s\geqslant2\) holds.
\end{proof}
The following theorem is an important consequence of previous discussions.
\begin{theorem}
If \(\BR_{\mathbb{P}^{n-1}}^*(\infty\times\infty\mapsto p,\frac{1}{2})<\infty\) for \(\frac{2n}{n-1}<p\leqslant4\) when \(n\geqslant 3\), or \(p>4\) when \(n=2\), then the linear restriction estimate \(\mathcal{R}_{\mathbb{P}^{n-1}}^*(\infty\mapsto p)\) holds.[3, p82-p84]
\end{theorem}
\begin{proof}
Let \(f:[-1,1]^{n-1}\to\mathbb{C}\) and \(\mathcal{C}\) as in Theorem 3.6, then
\begin{align*}
\mathcal{E}f(x)^2&=\int_{[-1,1]^{n-1}\times[-1,1]^{n-1}}f(\xi_1)f(\xi_2)e^{2\pi i((\xi_1+\xi_2)\cdot\overline{x}+(|\xi_1|^2+|\xi_2|^2)x_n)}d\xi_1d\xi_2,&&\\
&=\sum_{\Omega=\Omega_1\times\Omega_2\in\mathcal{C}}\int_{\Omega_1\times\Omega_2}f(\xi_1)f(\xi_2)e^{2\pi i((\xi_1+\xi_2)\cdot\overline{x}+(|\xi_1|^2+|\xi_2|^2)x_n)}d\xi_1d\xi_2,&&\\
&=\sum_{\Omega\in\mathcal{C}}\mathcal{E}_{\Omega_1}f(x)\mathcal{E}_{\Omega_2}f(x).
\end{align*}
For \(k\geqslant1\), define \(\mathcal{C}_k:=\{\Omega\in\mathcal{C}:l(\Omega)=2^{-k}\}\). By the triangle inequality,
\[\|\mathcal{E}f\|_{L^p}^2=\|(\mathcal{E}f)^2\|_{L^\frac{p}{2}}\leqslant\sum_{k\geqslant1}\|\sum_{\Omega\in\mathcal{C}_k}\mathcal{E}_{\Omega_1}f\mathcal{E}_{\Omega_2}f\|_{L^{\frac{p}{2}}}.\]
The sum set \(\Omega_1+\Omega_2\) is a cube of length \(2l(\Omega)\) centred on the sum of the centres of \(\Omega_1,\Omega_2\). As \(d(\Omega_1,\Omega_2)\leqslant100l(\Omega)\) by Theorem 3.6, this forces \(\Omega_1+\Omega_2\subset1001\Omega_1\). Also each \(\Omega_1\) appears at most \(O(1)\) times as the first component of some \(\Omega\in\mathcal{C}_k\) by the construction of \(\mathcal{C}\). Ranging \(\Omega\) through \(\mathcal{C}_k\), the collection of cubes \(4(\Omega_1+\Omega_2)\) has side length \(8l(\Omega)\) and overlap at most \(C\) times, \(C\) independent of \(k\). 

Note that 
\[ \supp(\widehat{\mathcal{E}_{\Omega_1}f\mathcal{E}_{\Omega_2}f})=\supp(\widehat{\mathcal{E}_{\Omega_1}f}*\widehat{\mathcal{E}_{\Omega_2}f})=\supp(\widehat{\mathcal{E}_{\Omega_1}f})+\supp(\widehat{\mathcal{E}_{\Omega_2}f})=R_\Omega,\]
where \(R_\Omega\subset\mathbb{R}^{n-1}\times\mathbb{R}\) is a rectangular box whose projection to \(\mathbb{R}^{n-1}\) lies inside \(2(\Omega_1+\Omega_2)\). Then \(\mathcal{C}_k\) can be split into \(C=O(1)\) families, such that the boxes \(2R_\Omega\) are pairwise disjoint for \(\Omega\) in each family. By Theorem 3.7, let \(s=\frac{p}{2}\), then

\(\|\sum_{\Omega\in\mathcal{C}_k}\mathcal{E}_{\Omega_1}f\mathcal{E}_{\Omega_2}f\|_{L^\frac{p}{2}}\leqslant(\sum_{\Omega\in\mathcal{C}_k}\|\mathcal{E}_{\Omega_1}f\mathcal{E}_{\Omega_2}f\|_{L^\frac{p}{2}}^\frac{p}{2})^\frac{2}{p}\) for \(n\geqslant3\), and

\(\|\sum_{\Omega\in\mathcal{C}_k}\mathcal{E}_{\Omega_1}f\mathcal{E}_{\Omega_2}f\|_{L^\frac{p}{2}}\leqslant(\sum_{\Omega\in\mathcal{C}_k}\|\mathcal{E}_{\Omega_1}f\mathcal{E}_{\Omega_2}f\|_{L^\frac{p}{2}}^\frac{p}{p-2})^\frac{p-2}{p}\) for \(n=2\).

Since the distance \(D\) between the centres of \(\Omega_1,\Omega_2\) satisfies \(D\geqslant d(\Omega_1,\Omega_2)\geqslant4l(\Omega)\), then by Theorem 3.4, 
\begin{align*}
\|\mathcal{E}_{\Omega_1}f\mathcal{E}_{\Omega_2}f\|_{L^\frac{p}{2}}^\frac{1}{2}&=\||\mathcal{E}_{\Omega_1}f\mathcal{E}_{\Omega_2}f|^\frac{1}{2}\|_{L^p},&&\\
&\lesssim D^{n-1-\frac{n+1}{p}}\|f\|_{L^\infty([-1,1]^{n-1})},&&\\
&\lesssim 2^{-k(n-1-\frac{n+1}{p})}\|f\|_{L^\infty([-1,1]^{n-1})}.
\end{align*}
Each \(\Omega_1\) appears at most \(O(1)\) times as the first component of some \(\Omega\in\mathcal{C}_k\), hence there are \(2^{k(n-1)}O(1)\) cubes in \(\mathcal{C}_k\), hence

\(\|\sum_{\Omega\in\mathcal{C}_k}\mathcal{E}_{\Omega_1}f\mathcal{E}_{\Omega_2}f\|_{L^\frac{p}{2}}\lesssim2^{-k((1-n)+p(n-1-\frac{n+1}{p}))\frac{2}{p}}\|f\|_{L^\infty([-1,1]^{n-1})}^2\) for \(n\geqslant3\), and 

\(\|\sum_{\Omega\in\mathcal{C}_k}\mathcal{E}_{\Omega_1}f\mathcal{E}_{\Omega_2}f\|_{L^\frac{p}{2}}\lesssim2^{-k(\frac{2p}{p-2}(1-\frac{3}{p})-1)\frac{p-2}{p}}\|f\|_{L^\infty([-1,1]^{n-1})}^2\) for \(n=2\).

This implies that \(\|\sum_{\Omega\in\mathcal{C}_k}\mathcal{E}_{\Omega_1}f\mathcal{E}_{\Omega_2}f\|_{L^\frac{p}{2}}\) is at most of \(O(2^{-k\varepsilon_p}\|f\|_{L^\infty([-1,1]^{n-1})}^2)\), where \(\varepsilon_p>0 \) for all \(n\geqslant2\) since \(p>\frac{2n}{n-1}\). Using previous results,
\[\|\mathcal{E}f\|_{L^p}^2\lesssim\sum_{k\geqslant1}2^{-k\varepsilon_p}\|f\|_{L^\infty([-1,1]^{n-1})}^2\lesssim\|f\|_{L^\infty([-1,1]^{n-1})}^2,\]
which is the linear restriction estimate  \(\mathcal{R}_{\mathbb{P}^{n-1}}^*(\infty\mapsto p)\).
\end{proof}
With the theorem above and the bilinear restriction theorem, then the linear restriction estimate \(\mathcal{R}_{\mathbb{P}^{n-1}}^*(\infty\mapsto p)\) holds for all \(p>\frac{2(n+2)}{n}\), and for \(n=2\), the full linear restriction range in Conjecture 2 is proved. 

Although the method of finding linear restriction range from bilinear restriction estimate only verifies the full linear restriction range on the paraboloid when \(n=2\), and the amount of work done is far more than the direct discussion of the linear restriction range on 2-dimensional curves with non-zero curvatures in the previous sections, in which the paraboloid falls, using bilinear restriction estimates provides a general way to find linear restriction estimates on other manifolds, such as the truncated cone, which is not differentiable, as long as the Restriction Theorem 4 and theorem 3.8 holds. In this way, any improvements of the bilinear restriction range of dimension greater than \(2\) can be used to refine the linear restriction range.

\section{Multilinear Restriction}
Having discussed the bilinear restriction problem in the previous section, there are three aspects to which the theory can be generalized: the dimension, the chosen manifold, and the choice of \(k\)-linearity, in which the previous section selects the 1-dimensional paraboloid, and \(k=2\). 

For an arbitrary \(1\leqslant k\leqslant n\), and the manifold is an \(n-1\) dimensional manifold, the \(k\)-linear restriction theory requires the assumption on the manifold to have at least \(n-k\) non-zero principal curvatures, which can be seen in section 2.[12] This implies that for \(n-1\) dimensional manifolds, only the \(n\)-linear restriction estimate does not require assumptions on curvatures.

For each \(1\leqslant j\leqslant n\), let \(\mathcal{M}^{\psi_j}:=\{\xi,\psi_j(\xi):\xi\in U_j\}\) be a sequence of smooth \(n-1\) dimensional manifolds, with the additional assumption on smoothness: there exists a constant \(A>0\) with
\[\|\psi_j\|_{C^2(U_j)}\leqslant A.\]
Also assume the transverality condition:
\begin{definition}[\(\nu\)-transversality]
 Let \(N_j(\xi_j)\) be a unit normal of \(\mathcal{M}^{\psi_j}\) on \(\xi_j\in\mathcal{M}^{\psi_j}\), \(1\leqslant j\leqslant n\), the collection of manifolds are \(\nu\)-transverse if there exists a constant \(\nu>0\) such that
\[\det(N_1(x_j),\cdots,N_n(x_n))\geqslant\nu\]
for all \(x_1\in U_1,\cdots,x_n\in U_n\).
\end{definition}
The \(\nu\)-transverse condition ensures the normals span the whole \(\mathbb{R}^n\). Note that in the previous section, the disjointness of the cubes \(\Omega_1\), \(\Omega_2\) ensure their respective paraboloids are transverse.

With the two assumptions, it is conjectured that 
\begin{conjecture}[Multilinear Restriction]
If the transversality and smoothness condition hold, for all \(f_j\in L^p(\mathcal{M}^{\psi_j})\), the inequality
\[\|\prod_{j=1}^n\mathcal{E}^{{\psi_j}}f_j\|_{L^\frac{q}{n}(\mathbb{R}^n,dx)}\lesssim_{A,\nu,n}\prod_{j=1}^n\|f_j\|_{L^p(U_j)}\]
holds for \(\frac{1}{q}<\frac{n-1}{2n}\) and \(p'\leqslant\frac{n-1}{n}q\).[12]
\end{conjecture}

Note that the bilinear restriction estimate from the previous section is the multilinear case when \(k=2\) and \(q=p=2\). 

Most ways of attempting to solve the conjecture require lengthy discussions of Kakeya-type estimates, and the multilinear restriction problem is solved by further discussions of the relationships between the multilinear Kakeya estimate and the multilinear restriction estimate, up to a small perturbation.

This section will discuss a shorter proof from [14], of a close result to the multilinear restriction problem, where the regularity of \(\psi\) is modified to \(\|\partial^\alpha\sum_j\|_{L^\infty}\lesssim_\alpha1\) . By multilinear interpolation and H\"older's inequality as in the previous section, it suffices to prove the endpoint result \(p=2\), \(q=\frac{2d}{d-1}\) to ensure the validity of the full restriction range. The rest of the section follows closely with [14], with additional discussions that are necessary.

\begin{RT}[Multilinear Restriction Theorem]
Assume that the conditions \(\|\partial^\alpha\sum_j\|_{L^\infty}\lesssim_\alpha1\) and transversality hold, then for all \(\varepsilon>0\), there exists a constant \(C(\varepsilon)\) such that
\[\|\prod_{j=1}^n\mathcal{E}^{\psi_j}f_j\|_{L^\frac{2}{n}(B(0,R))}\lesssim_{\psi,U,\nu,n}C(\varepsilon)R^\varepsilon\prod_{j=1}^n\|f_j\|_{L^p(U_j)}\]
holds. [12]
\end{RT}

An identity which will be used is the estimate of superposition of functions:
\begin{lemma}
   For \(f\in L^p\), for \(0<p\leqslant1\),
\[\|\sum_\alpha f_\alpha\|_{L^p}^p\leqslant\sum_\alpha\|f_\alpha\|_{L^p}^p.\]
[14]
\end{lemma}
\begin{proof}
This can be seen by showing \(\int|\sum_\alpha f_\alpha|^p\leqslant\sum_\alpha\int|f_\alpha|^p\) inductively. While the base case is trivial, let \(T(a):=(a+b)^p-a^p-b^p\) where \(a,b\) are non negative, then since \(0<p\leqslant1\), 
\[T'(a)=p((a+b)^{p-1}-a^{p-1})<0.\]
Hence \(T(a)\leqslant T(0)\), which implies \((a+b)^p\leqslant a^p+b^p\), and \(|\sum_\alpha |f_\alpha||^p\leqslant\sum_\alpha|f_\alpha|^p\) everywhere.[14]
\end{proof}
Denote \(\mathcal{H}_j\subset\mathbb{R}^{n+1}\) to be the hyperplane of the spatial domain passing through the origin with normal \(N_j\), and similarly, \(\mathcal{H}'_j\subset\mathbb{R}^{n+1}\) of the Fourier domain sharing the same normal. The index \(j\) is used for further indexing on multiple hypersurfaces.

Denote the Fourier transform \(\mathcal{F}_j:\mathcal{H}_j\to\mathcal{H}'_j\) and the inverse Fourier transform \(\mathcal{F}_j^{-1}:\mathcal{H}'_j\to\mathcal{H}_j\).

Denote the variables \(x\in\mathbb{R}^{n+1}\) in the spatial domain by \(x=(x_j,x')\), and similarly \(\xi=(\xi_j,\xi')\) in the Fourier domain, where \(x_j,\xi_j\) are coordinates along \(N_j\) and \(x',\xi'\) are coordinates along \(\mathcal{H}_j\). Since both \(x_j\) and \(\xi_j\) are constants, the operators \(\mathcal{F}_j,\mathcal{F}_J^{-1}\) only act on \(x',\xi'\) respectively. Further let \(\pi_{N_j}:\mathbb{R}^{n+1}\to\mathcal{H}_j\) be the canonical projection of \(\mathbb{R}^{n+1}\) onto \(\mathcal{H}_j\).

Assume \(U_1\subset\mathcal{H}_j\) is open and bounded. For \(f:U_j\to\mathbb{C}\), \(f\in L^2(U_j)\), define the extension operator \(\mathcal{E}_j:L^2(U_j)\to L^\infty(\mathbb{R}^{n+1})\) by 
\[\mathcal{E}_jf(x):=\int_{U_j}f(\xi')e^{2\pi i(x'\cdot\xi'+x_j\varphi_j(\xi'))}d\xi'\]
as in the previous section.

Due to the uncertainty principle, define the Fourier multiplier operator \(\nabla\varphi_j(\frac{D'}{i})\) to be the operator with symbol \(\nabla\varphi_j(\xi')\), that is
\[(\nabla\varphi_j(\frac{D'}{i})f)^{\wedge}(\xi')=\nabla\varphi_j(\xi')\widehat{f}(\xi').\]
Then 
\begin{lemma}
For any fixed \(x_0\in\mathbb{R}^{n+1}\), for all \(N\in\mathbb{N}\),
\[[x'-x_0'-x_j\nabla\varphi(\frac{D'}{i})]^N\mathcal{E}_jf=\mathcal{E}_j(\mathcal{F}_j((x'-x_0')^N\mathcal{F}_J^{-1}f)).\]
The power \(N\) refers to taking the N-fold dot product of the operator with itself, and \((x'-x_0')^N\) refers to the \(N\) times dot product of the vector \((x'-x_0')\) to itself.[14]
\end{lemma}\begin{proof}
For \(N=1\), 
\begin{align*}
&[x'-x_0'-x_j\nabla\varphi(\frac{D'}{i})]\mathcal{E}_jf&&\\
=&[x'-x_0'-x_j\nabla\varphi(\frac{D'}{i})]\int_{U_j}e^{2\pi i(x'\xi'+x_j\varphi_j(\xi'))}f(\xi')d\xi',&&\\
=&\int_{U_j}(x'-x_0'-x_j\nabla\varphi_j(\xi'))e^{2\pi i(x'\xi'+x_j\varphi_j(\xi'))}f(\xi')d\xi'.
\end{align*}
The case \(N>1\) is follows the same process using integration by parts \(N\) times.
\end{proof}
This implies that if \(\mathcal{F}_j^{-1}f\) is rapidly decaying outside of \(\{x':|x'-x_0'|\lesssim A\}\), then \(\mathcal{E}_jf\) decays rapidly outside of \(\{x':|x'-x_0'-x_j\nabla\varphi_j(\xi')|\lesssim A\}\), where \(\xi'\) covers \(\supp(f)\).

Let \(N_j\), \(j=1,\cdots,n+1\) be transversal unit vectors in \(\mathbb{R}^{n+1}\), and \(H_j\subset\mathbb{R}^{n+1}\) be the hyperplanes passing through the origin, with \(N_j\), \(j=1,\cdots,n+1\) the corresponding normals. By transversality, \(\{N_j,j=1,\cdots,n+1\}\) forms a basis of \(\mathbb{R}^{n+1}\), then let \(x\in\mathbb{R}^{n+1}\) with their coordinates defined with respect to the basis. Let \(\mathcal{L}:=\{z_1N_1+\cdots+z_{n+1}N_{n+1}:(z_1,\cdots,z_{n+1})\in\mathbb{Z}^{n+1}\}\) be the lattice in \(\mathbb{R}^{n+1}\) generated by \(\{N_j,j=1,\cdots,n+1\}\). Note that any lattice in \(\mathbb{R}^{n+1}\) has the same form as \(\mathcal{L}\) with possibly different basis vectors. Additionally, for each \(\mathcal{H}_j\), define the induced lattice \(\mathcal{L}(\mathcal{H}_j):=\pi_{N_j}(\mathcal{L})\). 

For \(r>0\), define \(\mathcal{C}(r)\) to be the set of parallelepipeds of size \(r\) in \(\mathbb{R}^{n+1}\) relative to the lattice \(\mathcal{L}\), which has the form
\[q(k):=[r(k_1-\frac{1}{2}),r(k_1+\frac{1}{2})]\times\cdots\times[r(k_{n+1}-\frac{1}{2}),r(k_{n+1}+\frac{1}{2})],\]
where \(k=(k_1,\cdots,k_{n+1})\in\mathcal{L}\). The "size" here implies the size of side lengths with respect to the unit vector.

Further define its centre \(c(q)=rk:=(rk_1,\cdots,rk_{n+1})\in r\mathcal{L}\). Then for each \(j=1,\cdots,n+1\), let \(\mathcal{CH}_j(r)=\pi_{N_j}\mathcal{C}(r)\) be the set of parallelepipeds of size \(r\) in \(\mathcal{H}_j\). For each \(q,q'\) in either \(\mathcal{CH}_j(r)\) or \(\mathcal{C}(r)\), define \(d(q,q')\) to be the distance between the two parallelepipeds.

Let \(\chi_0^n:\mathbb{R}^n\to[0,+\infty)\) be a Schwartz function such that \(\|\chi_0^n\|_{L^1}=1\) and \(\supp(\chi_0^n)\subset B(0,1)\). For each \(j\in\{1,\cdots,n+1\}\) and \(r>0\), define the linear operator \(\mathcal{T}_j:\mathcal{H}_j\to\mathbb{R}^n\) that maps \(\mathcal{L}(\mathcal{H}_j)\) to the standard lattice \(\mathbb{Z}^n\) in \(\mathbb{R}^n\). Then for each \(q\in\mathcal{CH}_j(r)\), define \(\chi_q:\mathcal{H}_j\to\mathbb{R}\) by 
\[\chi_q(x)=\chi_0^n(\mathcal{T}_j(\frac{x-c(q)}{r})).\]
Since \(\chi_0^n\) is a Schwartz function, which satisfies the criterion of the Poisson summation formula, then 
\[\sum_{q\in\mathcal{CH}_j(r)}\chi_q=\sum_{y\in\mathbb{Z}^n}\widehat{\chi_0^n}(y)e^{2\pi ix\cdot y}.\]
[14]

Since \(\supp(\widehat{\chi_0^n})\subset B(0,1)\), then \(\sum_{y\in\mathbb{Z}^n}\widehat{\chi_0^n}(y)e^{2\pi ix\cdot y}=\widehat{\chi_0^n}(0)\). Note that \(\chi_0^n\) is non-negative, and \(\widehat{\chi_0^n}(0)=\int\chi_0^n(x)e^{-2\pi ix\cdot0}dx=\|\chi_0^n\|_{L^1}=1\). Summing up the equalities:
\begin{lemma}
\[\sum_{q\in\mathcal{CH}_j(r)}\chi_q=1.\]
[14]
\end{lemma}
Define the operator \(\langle\cdot\rangle:=(1+|\cdot|^2)^\frac{1}{2}\), then using previous results, for each \(N\in\mathbb{N}\),
\begin{lemma}
\[\sum_{q\in\mathcal{CH}_j(r)}\|\langle\frac{x-c(q)}{r}\rangle^N\chi_qg\|_{L^2}^2\leqslant\|g\|_{L^2}^2.\]
[14]
\end{lemma}
\begin{proof}
\[\sum_{q\in\mathcal{CH}_j(r)}\|\langle\frac{x-c(q)}{r}\rangle^N\chi_qg\|_{L^2}^2\leqslant\int|\sum_{q\in\mathcal{CH}_j(r)}\langle\frac{x-c(q)}{r}\rangle^N\chi_qg|^2.\]
Since \(\langle\frac{x-c(q)}{r}\rangle^N\) is a polynomial, then \(\langle\frac{x-c(q)}{r}\rangle^N\chi_q\) is also Schwartz. By the previous property, \(\sum_{q\in\mathcal{CH}_j(r)}\langle\frac{x-c(q)}{r}\rangle^N\chi_q\lesssim_N1\), which implies
\[\int|\sum_{q\in\mathcal{CH}_j(r)}\langle\frac{x-c(q)}{r}\rangle^N\chi_qg|^2\lesssim_N\|g\|_{L^2}^2\]
for all \(g\in L^2(\mathcal{H}_j)\).
\end{proof}
Using the previous notations, the discrete Loomis-Whitney inequality states
\begin{theorem}[discrete Loomis-Whitney inequality]
\[\|\prod_{j=1}^{n+1}g_j(\pi_{N_j}(z))\|_{l^\frac{2}{n}(\mathcal{L})}\lesssim\prod_{j=1}^{n+1}\|g\|_{l^2(\mathcal{L}(\mathcal{H}_j))}.\]
[14]
\end{theorem}
For some \(0<\delta\ll1\), split each \(U_j\), \(j=1,\cdots,n+1\) into smaller pieces of diameter at most \(\delta\), where diameter of \(U_j\) is the longest distance between two points in \(U_j\). The surfaces \(\Sigma_j(U_j)\) are also split into corresponding pieces. 

With some abuse of notation, let \(U_j\) be a piece of diameter as above. For each \(\Sigma_j(U_j)\), let \(\zeta_j^0\in\Sigma_j(U_j)\) and \(N_j=N_j(\zeta_j^0)\) be the normal to \(\Sigma_j\). Let \(\mathcal{H}_j\) be the transversal hyperplane passing through the origin with normal \(N_j(\zeta_j^0)\). By a change of coordinates, without loss of generality assume \(U_j\subset B_j(0,\delta)\subset\mathcal{H}_j\), where \(B_j(0,\delta)\) is the ball in the hyperplane \(\mathcal{H}_j\). Then the corresponding extension operator is
\[\mathcal{E}_jf_j(x):=\int_{U_j}f(\xi')e^{2\pi i(x'\cdot\xi'+x_j\varphi_j(\xi'))}d\xi',\]
where \(x=(x_j,x')\), \(x_j\) is the coordinate in the direction of \(N_j\) and  \(x'\) is the coordinate in the direction of \(\mathcal{H}_j\).

Since \(\Sigma_j\) is smooth and bounded and hence \(\varphi_j\) is Lipschitz continuous, then \(|\nabla\varphi_j(x)-\nabla\varphi_j(y)|\lesssim\delta\) for any \(x,y\in U_j\).

Similar to other methods that use Kakeya type estimates, the method to show Restriction Theorem 5 requires a continuous inductive argument, in the sense of estimating \(\prod_{j=1}^{n+1}\mathcal{E}_Jf_j\) on parallelepipeds that increase in size over time.

\begin{definition}[Margin]
For \(f:\mathcal{H}_j\to\mathbb{C}\), define the margin 
\[\margin^j(f):=dist(\supp(f),B_j(0,\delta)^c),\]
\(j=1,\cdots,n+1\).[14]
\end{definition}
As in the previous section, define the multilinear restriction constant:
\begin{definition}[Multilinear Restriction Constants]
For \(R\geq\delta^{-2}\), define \(\MR(R)\) be the smallest constant such that 
\[\|\prod_{j=1}^{n+1}\mathcal{E}_jf_j\|_{L^\frac{2}{n}(Q)}\leqslant \MR(R)\prod_{j=1}^{n+1}\|f_j\|_{L^2}\]
 holds for all parallelepipeds \(Q\in\mathcal{C}(R)\), with\(f_j\) satisfies \(\margin^j(f_j)\geqslant\delta-R^{-\frac{1}{2}}\).[14]
\end{definition}
Since  \(R\geq\delta^{-2}\), then \(R^{-\frac{1}{2}}\leqslant\delta\ll1\) and \(\delta-R^{-\frac{1}{2}}\geq0\).

The estimate inside any parallelepiped \(Q\in \mathcal{C}(\delta^{-1}R)\) can be evaluated by prior information on the estimates inside parallelepipeds \(q \in \mathcal C ( R ), q \cap Q \neq \emptyset\) with centres in \(R\mathcal{L}\). Without loss of generality, assume \(Q\) is centred at the origin. Using the projection \(\pi_{N_j}\) onto \(\mathcal{H}_j\), then \(\pi_{N_j}q\in\mathcal{CH}_j(R)\).

Each \(q \in \mathcal C ( R ), q \cap Q \neq \emptyset\) has size \(R\). The induction hypothesis is that if \(\margin^j(f_j) \geq \delta-R^{-1/2}\) the estimate
\begin{align*}
&\|\prod_{j=1}^{n+1}\mathcal{E}_jf_j\|_{L^\frac{2}{n}(q)}&&\\
\lesssim& \MR(R)\prod_{j=1}^{n+1}(\sum_{q'\in\mathcal{CH}_j(R)}\langle\frac{d(\pi_{N_j},q')}{R}\rangle^{-(2N-n^2)}\|\langle\frac{x-c(q')}{R}\rangle^N\chi_{q'}\mathcal{F}_j^{-1}f_j\|_{L^2}^2)^\frac{1}{2}
\end{align*}
holds.

Fix \(j=1\) and \(q'\in\mathcal{CH}_1(R)\), for \(x=(x_1,x')\), then we have the inequalities
\begin{align*}
&\|(x'-c(q')-x_1\nabla\varphi_1(\xi_0'))\mathcal{E}_1\mathcal{F}_1(\chi_{q'}\mathcal{F}_1^{-1}f_1)\cdot\prod_{j=2}^{n+1}\mathcal{E}_jf_j\|_{L^\frac{2}{n}(q)}&&\\
&\leqslant\|(x'-c(q')-x_1\nabla\varphi_1(\xi'))\mathcal{E}_1\mathcal{F}_1(\chi_{q'}\mathcal{F}_1^{-1}f_1)\cdot\prod_{j=2}^{n+1}\mathcal{E}_jf_j\|_{L^\frac{2}{n}(q)}&\\
&\,\,\,\,\,\,\,\,+\|x_1(\nabla\varphi_1(\xi_0')-\nabla\varphi_1(\xi'))\mathcal{E}_1\mathcal{F}_1(\chi_{q'}\mathcal{F}_1^{-1}f_1)\cdot\prod_{j=2}^{n+1}\mathcal{E}_jf_j\|_{L^\frac{2}{n}(q)},&&\\
&=\|\mathcal{E}_1\mathcal{F}_1((x'-c(q'))\chi_{q'}\mathcal{F}_1^{-1}f_1)\cdot\prod_{j=2}^{n+1}\mathcal{E}_jf_j\|_{L^\frac{2}{n}(q)}&&\\
&\,\,\,\,\,\,\,\,+\|x_1\mathcal{E}_1\mathcal{F}_1((\nabla\varphi_1(\xi_0')-\nabla\varphi_1(\xi'))\chi_{q'}\mathcal{F}_1^{-1}f_1)\cdot\prod_{j=2}^{n+1}\mathcal{E}_jf_j\|_{L^\frac{2}{n}(q)},&&\\
&\leqslant \MR(R)\|(x'-c(q'))\chi_{q'}\mathcal{F}_1^{-1}f_1\|_{L^2}\prod_{j=2}^{n+1}\|f_j\|_{L^2}&&\\
&\,\,\,\,\,\,\,\,+\MR(R)\delta^{-1}R\|(\nabla\varphi_1(\xi_0')-\nabla\varphi_1(\xi'))\chi_{q'}\mathcal{F}_1^{-1}f_1\|_{L^2}\prod_{j=2}^{n+1}\|f_j\|_{L^2},
\end{align*}\begin{align*}
&\leqslant \MR(R)(\|(x'-c(q'))\chi_{q'}\mathcal{F}_1^{-1}f_1\|_{L^2}+R\|\chi_{q'}\mathcal{F}_1^{-1}f_1\|_{L^2})\prod_{j=2}^{n+1}\|f_j\|_{L^2},&&\\
&\lesssim R\cdot \MR(R)\|\langle\frac{x'-c(q')}{R}\rangle\chi_{q'}\mathcal{F}_1^{-1}f_1\|_{L^2}\prod_{j=2}^{n+1}\|f_j\|_{L^2}.
\end{align*}
Here the first step is a triangle inequality that separates \(x_1\nabla\varphi_1(\xi_0')=x_1(\nabla\varphi_1(\xi_0')+\nabla\varphi_1(\xi')-\nabla\varphi_1(\xi'))\), where \(\xi_0'\) is fixed. The second step is a consequence of lemma 4.2 and \(\nabla\varphi_1(\xi_0')-\nabla\varphi_1(\xi')\) can be seen as a constant with respect of \(\mathcal{E}_1\mathcal{F}_1\). The third step uses the induction hypothesis and the \(|x_1|\lesssim\delta^{-1}R\) in \(Q\). The final step follows by the definition of \(\langle\cdot\rangle\).

The support of \(f_j\) is required to be slightly altered to ensure the validity of the previous induction estimate. According to Bejenaru's work, by assuming \(\margin^j (f_j) \geq \delta - (\delta^{-1} R)^{-1/2}\), the induction hypothesis holds for \(R\) being updated to \(\delta^{-1} R\). Applying this modification while inducting does not fail the argument, as the accumulation is convergent.

For any \(q\in\mathcal{C}(R)\cap Q\) and \(x'\in\pi_{N_1}(q)\), since \(|x_1|\lesssim\delta^{-1}R\) and \(|\nabla\varphi_1(\xi_0')|\leqslant\delta\), then \(|x_1\nabla\varphi_1(\xi_0')|\leqslant R\) and
\[\langle\frac{x'-c(q')-x_1\nabla\varphi_1(\xi_0')}{R}\rangle\sim\langle\frac{d(\pi_{N_1}q,q')}{R}\rangle.\]
Hence by the previous two estimates,
\begin{align*}
&\|\mathcal{E}_1\mathcal{F}_1(\chi_{q'}\mathcal{F}_1^{-1}f_1)\cdot\prod_{j=2}^{n+1}\mathcal{E}_jf_j\|_{L^\frac{2}{n}(q)}&&\\
\lesssim& \MR(R)\langle\frac{d(\pi_{N_1}q,q')}{R}\rangle^{-1}\|\langle\frac{x'-c(q')}{R}\rangle\chi_{q'}\mathcal{F}_1^{-1}f_1\|_{L^2}\prod_{j=2}^{n+1}\|f_j\|_{L^2}.
\end{align*}
Repeating the argument for \(N\) times as in Lemma 4.2,
\begin{align*}
&\|\mathcal{E}_1\mathcal{F}_1(\chi_{q'}\mathcal{F}_1^{-1}f_1)\cdot\prod_{j=2}^{n+1}\mathcal{E}_jf_j\|_{L^\frac{2}{n}(q)}&&\\
\lesssim_N& \MR(R)\langle\frac{d(\pi_{N_1}q,q')}{R}\rangle^{-N}\|\langle\frac{x'-c(q')}{R}\rangle^N\chi_{q'}\mathcal{F}_1^{-1}f_1\|_{L^2}\prod_{j=2}^{n+1}\|f_j\|_{L^2}.
\end{align*}
Hence by the previous inequality and Lemma 4.3 and Lemma 4.1,
\begin{align*}
&\|\mathcal{E}_1f_1\prod_{j=2}^{n+1}\mathcal{E}_jf_j\|_{L^\frac{2}{n}(q)}^\frac{2}{n}&&\\
&\leqslant \MR(R)^\frac{2}{n}\sum_{q'\in\mathcal{CH}_1(R)}\|\mathcal{E}_1\mathcal{F}_1(\chi_{q'}\mathcal{F}_1^{-1}f_1)\cdot\prod_{j=2}^{n+1}\mathcal{E}_jf_j\|_{L^\frac{2}{n}(q)}^\frac{2}{n},&&\\
&\lesssim_N\MR(R)^\frac{2}{n}(\sum_{q'\in\mathcal{CH}_1(R)}\langle\frac{d(\pi_{N_1}q,q')}{R}\rangle^{-\frac{2N}{n}}\|\langle\frac{x'-c(q')}{R}\rangle^N\chi_{q'}\mathcal{F}_1^{-1}f_1\|_{L^2}^{\frac{2}{n}})\prod_{j=2}^{n+1}\|f_j\|_{L^2}^\frac{2}{n},&&\\
&\lesssim_N\MR(R)^\frac{2}{n}\prod_{j=2}^{n+1}\|f_j\|_{L^2}^\frac{2}{n}(\sum_{q'\in\mathcal{CH}_1(R)}\langle\frac{d(\pi_{N_1}q,q')}{R}\rangle^{-(2N-n^2)}\|\langle\frac{x'-c(q')}{R}\rangle^N\chi_{q'}\mathcal{F}_1^{-1}f_1\|_{L^2}^2)^\frac{1}{n},
\end{align*}
where the last step follows by \(\|a_j\cdot b_j\|_{l_j^\frac{2}{n}}\lesssim\|a_j\|_{l_j^2}\|b_j\|_{l_j^\frac{2}{n-1}}\) and 
\[\|\langle\frac{d(\pi_{N_1}q,q')}{R}\rangle^{-\frac{n^2}{2}}\|_{l_{q'}^\frac{2}{n-1}}\lesssim1.\]
The previous series of estimates are the argument for \(f_1\), repeating the same argument for \(f_2\cdots f_{n+1}\) and applying Lemma 4.4, then
\begin{align*}
&\|\prod_{j=1}^{n+1}\mathcal{E}_jf_j\|_{L^\frac{2}{n}(q)}&&\\
\lesssim& \MR(R)\prod_{j=1}^{n+1}(\sum_{q'\in\mathcal{CH}_j(R)}\langle\frac{d(\pi_{N_j},q')}{R}\rangle^{-(2N-n^2)}\|\langle\frac{x-c(q')}{R}\rangle^N\chi_{q'}\mathcal{F}_j^{-1}f_j\|_{L^2}^2)^\frac{1}{2}.
\end{align*}
Define the functions \(g_j:\mathcal{L}(\mathcal{H}_j)\to\mathbb{R}\) by
\[g_j(k)=(\sum_{q'\in\mathcal{CH}_j(R)}\langle\frac{d(q(k),q')}{R}\rangle^{-(N-2n^2)}\|\langle\frac{x'-c(q')}{R}\rangle^N\chi_{q'}\mathcal{F}_n^{-1}f_j\|_{L^2}^2)^\frac{1}{2},\,\,\,\,k\in\mathcal{L}(\mathcal{H}_j).\]
By Lemma 4.4, for \(N=N(n)\) large enough, \(g_j\in l^2(\mathbb{Z}^n)\) and 
\[\|g_j\|_{l^2(\mathcal{L}(\mathcal{H}_j))}\lesssim\|f_j\|_L^2.\]
By the discrete Loomis-Whitney inequality,
\[\|\prod_{j=1}^{n+1}\mathcal{E}_jf_j\|_{L^\frac{2}{n}(Q)}\lesssim \MR(R)\prod_{j=1}^{n+1}\|f_j\|_{L^2},\]
and hence \(\MR(\delta^{-1}R)\leqslant C\MR(R)\), where \(C\) is independent of \(\delta\) and \(R\). Iterating \(N\) times, \(\MR(\delta^{-N}r)\leqslant C^N\MR(R)\).

By the pointwise uniform bound
\[\|\prod_{j=1}^{n+1}\mathcal{E}_jf_j\|_{L^\infty}\lesssim\prod_{j=1}^{n+1}\|\mathcal{E}_jf_j\|_{L^\infty}\lesssim\prod_{j=1}^{n+1}\|f_j\|_{L^2}.\]
Integrating over arbitrary cubes of size \(\leqslant\delta^{-2}\), 
\[\max_{r\in[0,\delta^{-2}]}\MR(\delta^{-N}r)\leqslant C^N\max_{r\in[0,\delta^{-2}]}\MR(r)=C^NC(\delta),\]
and for \(R\in[\delta^{-N},\delta^{-N-1}]\),
\[A(R)\leqslant C^NC(\delta)\leqslant R^\varepsilon C(\delta)\]
as long as \(C^N\leqslant\delta^{-N\varepsilon}\). Hence by letting \(\delta=C^{-\frac{1}{\varepsilon}}\),
\[\|\prod_{j=1}^n\mathcal{E}_jf_j\|_{L^\frac{2}{n}(B(0,R))}\lesssim C(\varepsilon)R^\varepsilon\prod_{j=1}^n\|f_j\|_{L^2(U_j)}\]
for all \(f_j\in L^2(U_j),j=1,\cdots n+1\).[14]

By additional assumptions on the \(k\)-dimensional parallelpiped spanned by the normals \(\{N_j(x_j)\}\), and the support of \(\psi\), further refinements of the multilinear restriction theorem has been done by Bejenaru.[14] However, the epsilon loss \(C(\varepsilon)R^\varepsilon\) on the right hand side remains to be the greatest barrier in this subfield of harmonic analysis, before proving the whole restriction conjecture.

\pagebreak

{\Huge \bf Bibliography}

\bigskip


[1] G.B. Folland, Real Analysis: Modern Techniques and Their Applications. 2nd ed., Wiley, 1999, 370-374.
 
[2] M. Ziemke, Interpolation of Operators: Riesz-Thorin, Marcinkiewicz. \url{https://www.math.drexel.edu/faculty/mjz55/wp-content/uploads/sites/8/2017/01/at-interpol.pdf}, 2011.

[3] C. Demeter, Fourier Restriction, Decoupling, and Applications. Cambridge University Press, 2020.

[4] P. Mattila, Fourier Analysis and Hausdorff Dimension. Cambridge University Press, 2015.
 
[5] T. Tao, 247B, Notes 1: Restriction Theory. What’s New. \url{https://terrytao.wordpress.com/2020/03/29/247b-notes-1-restriction-theory/}, 2020.

[6] C. Muscalu, and W. Schlag, Classical and Multilinear Harmonic Analysis. Cambridge University Press, 2013, 295-300.

 [7] E.M. Stein, Harmonic Analysis (PMS-43), Volume 43: Real-Variable Methods, Orthogonality, and Oscillatory Integrals. (PMS-43). Princeton University Press, 2016.

[8] L. Guth, The Polynomial Method notes. \url{https://math.mit.edu/\~lguth/PolyMethod/lect30.pdf}, 2012.

[9] V.I. Bogachev, Measure Theory. Springer, 2007.
 
[10] D. Kriventsov, The Restriction Problem and the Tomas-Stein Theorem. 2009. \url{https://math.uchicago.edu/\~may/VIGRE/VIGRE2009/REUPapers/Kriventsov.pdf}, 2009

[11] R. Strichartz, Restrictions of Fourier Transforms to Quadratic Surfaces and Decay of Solutions of Wave Equations. Duke Mathematical Journal 44.3 (1977). \url{https://doi.org/10.1215/s0012-7094-77-04430-1}, 1977.

[12] J. Bennett, A. Carbery, T. Tao, On the Multilinear Restriction and Kakeya Conjectures. \textit{Acta Mathematica} 196.2 (2006): 261–302. \url{https://doi.org/10.1007/s11511-006-0006-4},  2006.

[13] T. Hytönen, et al, Analysis in Banach Spaces: Martingales and Littlewood-Paley Theory. Springer, 2016.
 
[14] I. Bejenaru, The Multilinear Restriction Estimate: A Short Proof and a Refinement. Math. Res. Lett. (2017): 1585–1603. \url{https://intlpress.com/site/pub/files/\_fulltext/journals/mrl/2017/0024/0006/MRL-2017-0024-0006-a001.pdf}, 2017.
\end{document}